\documentclass[11 pt]{amsart}
\usepackage[utf8x]{inputenc}
\usepackage{mathrsfs}
\usepackage{amsmath}
\usepackage{amssymb}
\usepackage{amsthm}
\usepackage[all]{xy}
\usepackage{amsmath,amssymb,amsfonts}
\usepackage{textcomp}
\newtheorem{defn}{Definition}
\newtheorem{thm}{Theorem}[section]
\newtheorem{prop}[thm]{Proposition}
\newtheorem{lemme}[thm]{Lemma}

\theoremstyle{remark}
\newtheorem*{remark}{Remark}

\DeclareMathOperator{\num}{num}
\DeclareMathOperator{\nd}{nd}
\DeclareMathOperator{\DIV}{div}
\DeclareMathOperator{\Ker}{Ker}
\DeclareMathOperator{\vol}{vol}
\DeclareMathOperator{\image}{Im}
\DeclareMathOperator{\adj}{adj}
\DeclareMathOperator{\ac}{ac}
\DeclareMathOperator{\sing}{sing}

\begin{document}
\title[A
Kawamata-Viehweg-Nadel type vanishing theorem]
{Numerical dimension and a
Kawamata-Viehweg-Nadel type vanishing theorem on compact Kähler manifolds}
\author{Junyan CAO}
\email{junyan.cao@ujf-grenoble.fr}
\address{Université de Grenoble I, Institut Fourier, 38402 Saint-Martin d'Hères, France}

\begin{abstract}
Let $X$ be a compact Kähler manifold and let $L$ be a
pseudo-effective line bundle on $X$ with singular metric $\varphi$.
We first define a notion of numerical dimension of the pseudo-effective pair $(L, \varphi)$
and then discuss the properties of it. 
We prove also a very general Kawamata-Viehweg-Nadel type vanishing theorem
on an arbitrary compact Kähler manifold.
\end{abstract}

\maketitle
\section{Introduction}

Let $X$ be a compact Kähler manifold and let $(L, \varphi)$ be a
pseudo-effective line bundle on $X$ 
(c.f Section 3 for the definition of a pseudo-effective pair $(L, \varphi)$).
H.Tsuji \cite{Tsu} has defined a notion of numerical dimension by an
algebraic method: 
\begin{defn}
Let $ X $ be a projective variety and $ (L, \varphi) $ a pseudo-effective line bundle.
One defines
$$\nu_{\num}(L,\varphi)=\sup\{\dim V\mid  V \text{ subvariety of }X
\text{ such that} $$
$$\varphi \text{ is well defined on }V \text{ and }
(V, L,\varphi) \text{ is big.}\}$$
\end{defn}

\hspace{-12pt}Here $ (V, L, \varphi) $ to be big means that there is a desingularization 
$\pi: \widetilde{V}\rightarrow V$
such that
$$\varlimsup\limits_{m\to \infty}\frac{h^{0}(\widetilde{V},m\pi^{*}(L)\otimes\mathcal{I}(m \varphi\circ\pi))}{m^{n}}> 0$$
where $ n $ is the dimension of $ V $.
\footnote{\cite{Tsu} proved that the bigness does not depend on the choice of
desingularization.}
\vspace{5 pt}

Since Tsuji's definition depends on the existence of subvarieties, 
it is more convenient to find an analytic definition if the base manifold is not
projective.
Following a suggestion of J-P. Demailly, we first define a notion of
numerical dimension $\nd (L,\varphi)$
(cf. Definition 4) 
for a pseudo-effective line bundle $(L, \varphi)$ on a manifold $X$ which is just assumed to be compact Kähler.
The definition involves a certain cohomological intersection product of positive currents, introduced in Section 2.
We then discuss the properties of $\nd (L,\varphi)$ in Section 3 and 4. The main
properties are as follows.
\begin{prop}
Let $(L,\varphi)$ be a pseudo-effective line bundle on a projective variety $X$
of dimension $n$ and $\nd(L,\varphi)=n$.
Then
$$\varliminf\limits_{m\to
\infty}\frac{h^{0}(X,mL\otimes\mathcal{I}(m\varphi))}{m^{n}}> 0 .$$
\end{prop}

\begin{prop}
Let $(L,\varphi)$ be a pseudo-effective line bundle on a projective variety $X$. 
Then 
$$ \nu_{\num} (L, \varphi) =\nd (L,\varphi) .$$
\end{prop}

Our main interest in this article is to prove a very general Kawamata-Viehweg-Nadel type vanishing theorem
on an arbitrary compact Kähler manifold. Our statement is as follows.
\begin{thm}
Let $(L, \varphi)$ be a pseudo-effective
line bundle on a compact Kähler manifold $X$ of dimension $n$. Then
$$ H^{p}(X, K_{X}\otimes L\otimes \mathcal{I}_{+}(\varphi))=0\qquad
\text{for any }  p\geq n-\nd (L,\varphi)+1 ,$$
where $\mathcal{I}_{+}(\varphi)$ is the upper semicontinuous variant of the multiplier ideal sheaf
associated to $\varphi$ (cf. \cite{FJ}).
\end{thm}

The organization of the article is as follows. In Section 2, we first recall some elementary results about the 
analytic multiplier ideal sheaves and 
define our cohomological product of positive currents by quasi-equisingular approximation.
In Section 3, using the product defined in Section 2, 
we give our definition of the numerical dimension $\nd (L,\varphi)$ of
pseudo-effective line bundles with singular metrics. 
The main goal of this section is to give an asymptotic estimate when $\nd(L,\varphi)= \dim X$. 
In section 4, we prove that our numerical dimension coincides with the
definition in \cite{Tsu} when $X$ is projective. 
We also give a numerical criterion of the numerical dimension and discuss a
relationship between the numerical dimension without multiplier ideal sheaves
and the numerical dimension defined here. 
In Section 5, we first give a quick proof of a Kawamata-Viehweg-Nadel type
vanishing theorem on projective varieties. 
We finally generalize the vanishing theorem on arbitrary compact Kähler
manifolds by the methods developed in \cite{DP} and \cite{Mou}.

\hspace{-12pt}\textbf{Acknowledgements:} I would like to thank Professor J-P.
Demailly for numerous ideas and suggestions for this article, 
and also for his patience and disponibility.

\section{Cohomological product of positive currents}

We first recall some basic definitions and results about quasi-psh
functions. 
Let $X$ be a complex manifold. 
We say that $\varphi$ is a psh function (resp. a quasi-psh function) on $X$, 
if 
$$i\partial\overline{\partial}\varphi\geq 0 , \hspace{30pt}( \text{resp. }
i\partial\overline{\partial}\varphi\geq -c\cdot\omega_{X} )$$ 
where $c$ is a positive constant and $\omega_{X}$ is a smooth hermitian metric on $X$. 
We say that a quasi-psh function $\varphi$ has analytic singularities, if
$\varphi$ is locally of the form
$$\varphi(z)=c\cdot \ln (\sum |g_{i}|^{2})+O(1)$$
with $c>0$ and $\{g_{i}\}$ are holomorphic functions. 
Let $\varphi, \psi$ be two quasi-psh functions. 
We say that $\varphi$ is less singular than $\psi$ if 
$$\psi\leq\varphi+C$$ 
for some constant $C$. 
We denote it $\varphi\preccurlyeq\psi$.

We now recall the analytic definition of multiplier ideal sheaves.
Let $\mathcal{I}(\varphi)$ be the multiplier ideal sheaves associated to the
quasi-psh function $\varphi$, i.e.
$$\mathcal{I}(\varphi)_{x}=\{f\in \mathcal{O}_{X}| \int_{U_{x}}
|f|^{2}e^{-2\varphi}< +\infty \} $$
where $U_{x}$ is some open neighborhood of $x$ in $X$
(cf. \cite{Dem} for a more detailed introduction to the concept of
multiplier ideal sheaf).
When $\varphi$ does not possess analytic singularities, 
we need to introduce the ``upper semicontinuous regularization'' of the multiplier ideal sheaf,
namely the ideal sheaf
$$\mathcal{I}_{+}(\varphi)=\lim_{\epsilon\rightarrow 0^{+}}\mathcal{I}((1+\epsilon)\varphi).$$
By the Notherian property of coherent ideal sheaves, there exists an $\epsilon > 0$ such that 
$$\mathcal{I}_{+}(\varphi)=\mathcal{I}((1+\epsilon')\varphi)
\qquad\text{for any }0< \epsilon'< \epsilon.$$
When $\varphi$ has analytic singularities, 
it is easy to see that $\mathcal{I}_{+}(\varphi)=\mathcal{I}(\varphi)$. 
Conjecturally we have the equality for all psh functions. 
\footnote{This equality is well known in dimension $1$ 
and is proved to be true in dimension 2 by Favre-Jonsson \cite{FJ}.
See \cite{DP} for more details about $\mathcal{I}_{+}(\varphi)$.}

\hspace{-12pt}\textbf{Important Convention:} 
When we talk about a line bundle $L$ on $X$, 
we always first implicitly fix a smooth metric $h_{0}$ on $L$. 
Given a singular metric $\varphi$ on $L$ or sometimes $\varphi$ for simplicity, 
we just means that the new metric on $L$ is given by $h_{0}e^{-\varphi}$.
Recall that the curvature of the metric $\varphi$ for $L$ is
$$\frac{i}{2\pi}\Theta_{\varphi}(L)=\frac{i}{2\pi}\Theta_{h_{0}}(L)+dd^{c}\varphi .$$
The pair $(L,\varphi)$ is said to be a pseudo-effective line bundle if
$\frac{i}{2\pi}\Theta_{\varphi}(L)\geq 0$ as a current.
\vspace{10 pt}

Let $\pi: \widetilde{X}\rightarrow X$ be a modification of a smooth variety $X$, 
and let $\varphi$, $\psi$ be two quasi-psh fuctions on $X$
such that
$\mathcal{I}(\varphi)\subset \mathcal{I}(\psi)$.
In general, this inclusion does not imply
$\mathcal{I}(\varphi\circ\pi)\subset \mathcal{I}(\psi\circ\pi).$
We thus need the following lemma.
\begin{lemme}
Let $E=\pi^{*}K_{X}-K_{\widetilde{X}}$. 
If $\mathcal{I}(\varphi)\subset \mathcal{I}(\psi)$, then 
$$\mathcal{I}(\varphi\circ\pi)\otimes\mathcal{O}(-E)\subset \mathcal{I}(\psi\circ\pi),$$
where the sheaf $\mathcal{O}(-E)$ is the germs of holomorphic functions $f$ such that $\DIV(f)\geq E$.  

\end{lemme}

\begin{proof}
It is known that $\mathcal{I}(\varphi\circ\pi)\subset \pi^{*}\mathcal{I}(\varphi)$
(cf. Proposition 14.3 of \cite{Dem}).
Then for any $f\in \mathcal{I}(\varphi\circ\pi)_{x}$, 
we have 
$$\int_{\pi(U_{x})}\pi_{*}(f)e^{-2\varphi}< +\infty ,$$
where $U_{x}$ is some open neighborhood of $x$ and $\pi(U_{x})$ is its image under $\pi$ 
which is not necessary open.
Combining with the condition $\mathcal{I}(\varphi)\subset \mathcal{I}(\psi)$, 
we get
$$\int_{\pi(U_{x})} \mid \pi_{*}f\mid^{2}e^{-2\psi}< +\infty . \leqno (1.1)$$
$(1.1)$ implies that
$$\int_{U_{x}}  \mid f \mid^{2} \mid J\mid^{2} e^{-2\psi\circ\pi}<
+\infty, \leqno (1.2)$$
where $J$ is Jacobian of $\pi$. 
Since $\mathcal{O}(-E)=J\cdot \mathcal{O}_{X}$, 
$(1.2)$ implies the lemma.
\end{proof}

Let $X$ be a compact Kähler manifold and let $T$ be a closed positive $(1,1)$-
current. 
It is well known that $T$ can be written as 
$$T=\theta+dd^{c}\varphi ,$$ 
where $\theta$ is a smooth $(1,1)$-closed form representing $[T]\in
H^{1,1}(X, \mathbb{R})$ and $\varphi$ is a quasi-psh function. 
Demailly's famous regularization theorem states that $\varphi$ can be
approximated by a sequence of quasi-psh functions with analytic singularities. 
We say that it is an analytic approximation of $\varphi$.
Among all these analytic approximations, we want to deal with those which keep
the information concerning the singularities of $T$. 
More precisely, we introduce the following definition.

\begin{defn}
Let $\theta+dd^{c}\varphi$ be a positive current, where $\theta$ is a
smooth form and $\varphi$ is a quasi-psh function on a compact Kähler manifold
$(X,\omega)$.
We say that $ \{\varphi_{k} \}_{k=1}^{\infty} $ is a quasi-equisingular
approximation of $\varphi$ for the current $\theta+dd^{c}\varphi$
if it satisfies the following conditions:

$(i)$ $ \{\varphi_{k} \}_{k=1}^{\infty}$ converge to
$ \varphi $ in $L^{1}$ topology and
$$\theta+dd^{c}\varphi_{k}\geq  - \tau_{k}\cdot \omega$$
for some constants $\tau_{k}\rightarrow 0$ as $k\rightarrow +\infty$.

$(ii)$ all $\varphi_{k}$ have analytic singularities and $\varphi_{k}\preccurlyeq\varphi_{k+1}$
for any $k$.

$(iii)$ For any $ \delta> 0$ and $m\in\mathbb{N}$, there exists $k_{0}
(\delta, m)\in \mathbb{N}$
such that
$$ \mathcal{I} (m (1 + \delta) \varphi_{k}) \subset \mathcal{I} (m \varphi)
\qquad\text{for every }k \geq k_{0} (\delta, m)$$
\end {defn}

\begin{remark}
By condition $(i)$, such type of approximation depends not only on $\varphi$ 
but also on the current $\theta+dd^{c}\varphi$.
\end{remark}

The existence of such quasi-equisingular approximations was proved in Theorem
2.2.1 of \cite{DPS 01} by a Bergman kernel method.
The choice of such approximations corresponds in some sense to the most singular
approximations asymptotically. The following proposition makes this assertion
more precise.

\begin {prop}
Let $\theta+dd^{c}\varphi_{1}, \theta+dd^{c}\varphi_{2}$ be two positive currents on a
compact Kähler
manifold $X$ as usual.
We assume that the quasi-psh function $\varphi_{2}$ is more singular than
$\varphi_{1}$.
Let $ \{\varphi_{i, 1} \}_{i=1}^{\infty} $ be an analytic approximation to
$\varphi_{1}$ and let $ \{\varphi_{i, 2} \}_{i=1}^{\infty} $ be a
quasi-equisingular
approximation to $\varphi_{2} $. 
For any closed smooth $(n-1,n-1)$-semi-positive form $u$,
we have
$$\lim\limits_{i\to \infty}\int_{X}(dd^{c}\varphi_{i,1})_{\ac}\wedge u\geq
\lim\limits_{i\to \infty} \int_{X}(dd^{c}\varphi_{i,2})_{\ac}\wedge u\leqno
(2.1)$$
where $(dd^{c}\varphi_{i,1})_{\ac}$ denotes the absolutely continuous part of the
current $dd^{c}\varphi_{i,1}$.
\end {prop}

\begin{proof}

It is enough to show that
$$\int_{X} (dd^{c}\varphi_{s,1})_{\ac} \wedge u \geq \lim \limits_{i \to \infty
} \int_ {X} (dd^{c}\varphi_{i,2})_{\ac} \wedge u \leqno (2.2)$$
for any $s\in \mathbb{N}$ fixed.
Since $ \{\varphi_{i, 2} \}_{i=1}^{\infty} $ is a quasi-equisingular
approximation to $\varphi_{2}$, 
for any $\delta> 0$ and $m \in \mathbb{N}$, 
there exists a $k_{0} (\delta, m)\in\mathbb{N} $ such that 
$$ \mathcal{I} (m (1 + \delta) \varphi_{k,2}) \subset \mathcal{I} (m
\varphi_{2})\qquad\text{for every } k \geq k_{0} (\delta,
m).\leqno (2.3)$$
Since $\varphi_{1}\preccurlyeq \varphi_{2}$ and 
$\varphi_{s,1}\preccurlyeq \varphi_{1}$,
$(2.3)$ implies that
$$\mathcal{I} (m (1 + \delta) \varphi_{k,2}) \subset \mathcal{I} (m \varphi_{s,1})\leqno (2.4)$$
for any $s\in \mathbb{N}$ fixed and $k \geq k_{0} (\delta,
m)$.

Using $(2.4)$, we begin to prove $(2.2)$. 
Let $\pi: \widehat{X}\rightarrow X$ be a log resolution of $ \varphi_{s,1} $.
We can thus assume that
$$dd^{c}\varphi_{s,1}\circ\pi=[F] + C^{\infty} ,$$ 
where $ F $ is a $ \mathbb{R}$-normal crossing divisor.
Using Lemma 2.1, $(2.4)$ implies that
$$\mathcal{I}(m(1+\delta)\varphi_{k,2}\circ\pi)\otimes \mathcal{O}(-J)\subset
\mathcal{I}(m\varphi_{s,1}\circ\pi)=\mathcal{O}(-\lfloor mF\rfloor)\leqno
(2.5)$$
for $k\geq k_{0}(\delta, m)$, where $J$ is the Jacobian of the blow up $\pi$.
Since $F$ is a normal crossing divisor, 
$(2.5)$ implies that  $m(1+\delta)dd^{c}\varphi_{k,2}\circ\pi+[J]-\lfloor mF\rfloor$ is
a positive current.
Then
$$\int_{\widehat{X}}(dd^{c}m(1+\delta)\cdot\varphi_{k,2}\circ\pi)_{\ac}\wedge u\leq
C+\int_{\widehat{X}}(dd^{c}m\cdot\varphi_{s,1}\circ\pi)_{\ac}\wedge u $$
for $k\geq k_{0}(\delta, m)$, where the constant $C$ is independent of $m$ and $k$.
Therefore, if $m\rightarrow +\infty$, we get
$$\int_{\widehat{X}}(dd^{c}\varphi_{k,2}\circ\pi)_{\ac}\wedge u\leq
O(\frac{1}{m})+ C_{1}
\delta+\int_{\widehat{X}}(dd^{c}\varphi_{s,1}\circ\pi)_{\ac}\wedge u \leqno
(2.6)$$
for $k\geq k_{0}(\delta, m)$ and a constant $C_{1}$ independent of $m$ and $k$.
Then
$$\int_{X}(dd^{c}\varphi_{k,2})_{\ac}\wedge u\leq
O(\frac{1}{m})+ C_{1}
\delta+\int_{X}(dd^{c}\varphi_{s,1})_{\ac}\wedge u $$
for $k\geq k_{0}(\delta, m)$.
Letting $ m \rightarrow + \infty $ and $ \delta \rightarrow 0$, 
we get
$$\lim\limits_{k\to \infty} \int_{X}(dd^{c}\varphi_{k,2})_{\ac}\wedge
u\leq\int_{X}(dd^{c}\varphi_{s,1})_{\ac}\wedge u .$$
Thus $(2.2)$ is proved.

\end{proof}

\begin{remark}
Assume that $\{\varphi_{i,1}\},
\{\varphi_{i,2}\}$ are two quasi-equisingular approximations to the same
$\varphi$.
Proposition 2.2 implies in particular that 
$$\lim\limits_{i\to \infty}\int_{X}(dd^{c}\varphi_{i,1})_{\ac}\wedge u=
\lim\limits_{i\to \infty} \int_{X}(dd^{c}\varphi_{i,2})_{\ac}\wedge u.$$
\end{remark}

Thanks to Proposition 2.2 and the above remark, 
we can define a related cohomological product of closed positive $(1,1)$-currents.

\begin{defn}
Let $T_{1},\cdots,T_{k}$ be closed
positive $(1,1)$-currents on a compact Kähler manifold. 
We write them by the potential forms
$T_{i}=\theta_{i}+dd^{c}\varphi_{i}$ as usual.
Let $\{\varphi_{i,j}\}_{j=1}^{\infty}$ be a quasi-equisingular approximation to
$\varphi_{i}$. 
Then we can define a product
$$\langle T_{1}, T_{2},\cdots,T_{k} \rangle$$
as an element in $H^{k,k}_{\geq 0}(X)$ such that for all $u\in H^{n-k,n-k}(X)$,
$$\langle T_{1}, T_{2},\cdots,T_{k} \rangle\wedge u$$
$$=\lim\limits_{j\to
\infty}\int_{X}(\theta_{1}+dd^{c}\varphi_{1,j})_{\ac}\wedge\cdots\wedge
(\theta_{k}+dd^{c}\varphi_{k,j})_{\ac}\wedge u$$ 
where $\wedge$ is the usual wedge product in cohomology.
\end{defn}

\begin{remark}
Thanks to Proposition 2.2 and its remark, 
the product defined here does not depend on the choice of quasi-equisingular approximation. 
Moreover the product here is smaller than the product defined by any other analytic approximations.
\end{remark}

\section{Numerical dimension}

Using the product defined in the last section, we can give our
definition of the numerical dimension. 

\begin{defn} 
Let $(L,\varphi)$ be a pseudo-effective line
bundle on a compact Kähler manifold $X$ such that $\frac{i}{2\pi}\Theta_{\varphi}(L)\geq 0$. 
We define the numerical dimension $ \nd (L, \varphi) $ as the largest $ v \in
\mathbb {N} $, 
such that $\langle (i\Theta_{\varphi})^{v}\rangle\neq 0$,
where the cohomological product $\langle (i\Theta_{\varphi})^{v}\rangle$ is the
$v$-fold product of $i\Theta_{\varphi}(L)$ defined in Definition 3.
\end{defn}

Let $(L,\varphi)$ be a pseudo-effective line bundle on $X$ of dimension $n$ and
$ \nd (L, \varphi) = n $. 
If moreover $ \varphi $ is a quasi-psh function with analytic singularities, 
it is not difficult to see that
$$\frac{h^{0}(X,mL\otimes\mathcal{I}(m\varphi))}{m^{n}}$$ 
admits a strictly positive limit
by using the Riemann-Roch formula. 
When $ \varphi $ is just a quasi-psh function, H.Tsuji conjectured in \cite{Tsu} that
$$\frac{h^{0}(X,mL\otimes\mathcal{I}(m\varphi))}{m^{n}}$$
admits also a strictly positive limit. 
The main goal of this section is to prove that if $\nd(L,\varphi)=n$, then
$$\varliminf\limits_{m\to \infty}\frac{h^{0}(X,mL\otimes\mathcal{I}(m\varphi))}{m^{n}}>0.$$
\vspace{10pt}

For the more precise estimates, 
we first need to explain the construction of quasi-equisingular approximations by
a Bergman Kernel method. 
Before doing this, we first prove a useful estimate which is essentially
proved in \cite{DP} in a more general situation. 
For the sake of completeness, we give the proof here with a more precise
estimate.

\begin{lemme}
Let $A$ be a sufficiently large very ample line bundle on a projective manifold $X$ 
and let $(L,\varphi)$ be a pseudo-effective line bundle. 
Let $\varphi_{m}$ be the metric on $L$ constructed by the Bergman Kernel of
$H^{0}(X, A+mL\otimes\mathcal{O}(m\varphi))$ with respect to the metric $m\varphi$.
Then
$$\mathcal{I}(\frac{sm}{m-s}\varphi_{m})\subset \mathcal{I}(s\varphi)\qquad
\text{for any}\hspace{5 pt} m,s\in \mathbb{N} .$$
\end{lemme}

\begin{proof}

First of all, we have the following estimate on $X$:
$$\int_{s\cdot\varphi(x)\leq\frac{sm}{m-s}\cdot\varphi_{m}(x)}
e^{-2s\cdot\varphi(x)}=\int_{s\cdot\varphi(x)\leq\frac{sm}{m-s}\cdot\varphi_{m}(
x)}
e^{2(m-s)\cdot\varphi(x)-2m\cdot\varphi(x)}$$
$$\leq \int_{X} e^{2m\cdot\varphi_{m}} e^{-2m\cdot\varphi} =h^{0}(X,
A+mL\otimes\mathcal{I}(m\varphi))< +\infty .$$
Using the above finiteness, for any
$f\in\mathcal{I}(\frac{sm}{m-s}\varphi_{m})_{x}$, we have
$$\int_{U_{x}} \mid f\mid^{2}e^{-2s\varphi}\leq
\int_{\varphi\leq\frac{sm}{m-s}\varphi_{m}}\mid
f\mid^{2}e^{-2s\varphi}+\int_{U_{x}}\mid
f\mid^{2}e^{-\frac{2sm}{m-s}\varphi_{m}}$$
$$\leq \sup \mid f\mid^{2}\cdot\int_{\varphi\leq\frac{sm}{m-s}\varphi_{m}}
e^{-2s\varphi}+\int_{U_{x}}\mid f\mid^{2}e^{-2\frac{sm}{m-s}\varphi_{m}}<+\infty
.$$ 
Then $f\in\mathcal{I}(s\varphi)$. The lemma is proved.
\end{proof}

We are going to construct a quasi-equisingular approximation to $\varphi$.
Such type of approximations were implicitly constructed
in\cite{DPS 01} for the local case.
Philosophically, 
we can generalize the local properties to the global case after tensoring by $ K_ {X} $. 
The Ohsawa-Takegoshi extension theorem and the Skoda theorem are two typical examples.
In fact, at least in the projective case,  when $ S $ is a smooth zero divisor of
a section of an ample line bundle then $ X \setminus S $ is Stein. 
So we can reduce the proof to the Stein case.
The reason for tensoring with $ K_{X} $ is that the $ L^{2} $ integral is
globally well defined after tensoring by $ K_{X} $.
\vspace{5 pt}

\begin{prop}
Let $ X $ be a projective variety of dimension $ n $ and let $\omega$ be
a Kähler metric in $H^{1,1}(X,\mathbb{Q})$.
Let $ (L, \varphi) $ 
(i.e. the metric on $L$ is $h_{0}e^{-\varphi}$ for some smooth metric $h_{0}$ and quasi-psh function $\varphi$) 
be a pseudo-effective line bundle on $X$ such that $ \nd
(L, \varphi) = n $.

Let $ (G , h_{G} )$ be an ample line bundle on $ X $ with smooth metric $h_{G}$, 
such that the curvature $i\Theta_{h_{G}}(G)$ is positive and sufficiently large (e.g. $G$ very ample and $G-K_{X}$ ample).  
Let $\{\tau_{p, q, i}\}_{i}$ be an orthonormal basis of 
$$H^{0}(X, 2^{p}G+2^{q}L\otimes\mathcal{I}(2^{q}\varphi))$$
with respect to the singular metric $h_{G}^{2^{p}}\cdot h_{0}^{2^{q}}\cdot e^{-2^{q}\varphi}$.
We define
$$\varphi_{p, q}
=\frac{1}{2^{q}}\ln \sum_{i} |\tau_{p, q , i}|^{2}_{h_{G}^{2^{p}}\cdot h_{0}^{2^{q}}}.$$
Then there exist two increasing integral sequences $p_{m}\rightarrow +\infty$ and $q_{m}\rightarrow +\infty$ with
$$\lim (q_{m} - p_{m})=+\infty \text{ and } q_{m}-q_{m-1}\geq p_{m}-p_{m-1} \text{ for all } m\in \mathbb{N}, $$
such that $\{ \varphi_{p_{m}, q_{m}}\}_{m=1}^{+\infty}$ is a quasi-equisingular approximation to $\varphi$ 
for the current 
$\frac{i}{2\pi}\Theta_{h_{0}}(L)+dd^{c}\varphi$.
We denote $\varphi_{p_{m}, q_{m}}$ by $\varphi_{m}$ for simplicity.

Moreover, $\{\varphi_{m}\}$ satisfies the following two properties:
$$H^{0}(X, 2^{p_{m}}G+2^{q_{m}}L\otimes\mathcal{I}(2^{q_{m}}\varphi_{m}))\leqno (i)$$
$$=H^{0}(X, 2^{p_{m}}G+2^{q_{m}}L\otimes\mathcal{I}(2^{q_{m}}\varphi))$$
for every $m\in \mathbb{N}^{+}$.

\hspace{-12pt}$(ii)$ There exists a constant $C>0$ independent of $G$, $ s_{0}$, such that 
for all $ \epsilon> 0$, we have
$$\int_{X}(\frac{i}{2\pi}\Theta_{\varphi_{m}}(L)+\epsilon\omega)_{\ac} ^{n}> C $$
for $ m\geq m_{0}(\epsilon)$.
\end{prop}

\begin{proof}

By Theorem 13.21, 13.23 of \cite{Dem}, 
there exists two squences $p_{m}\rightarrow +\infty$ and $q_{m}\rightarrow +\infty$ with
$$\lim q_{m}/ p_{m}=+\infty \text{ and } q_{m}-q_{m-1}\geq p_{m}-p_{m-1} \text{ for all } m\in \mathbb{N}, $$
such that $\{ \varphi_{m}\}$ is an analytic approximation to $\varphi$
for the current $\frac{i}{2\pi}\Theta_{h_{0}}(L)+dd^{c}\varphi$, i.e. 
it satisfies property $(i)$ in Definition 2.
Lemma $3.1$ implies that $\{ \varphi_{m}\}$ satisfies property $(iii)$ in Definition 2. 
To prove that $\{ \varphi_{m}\}$ is a quasi-equisingular approximation, 
it is enough to prove property $(ii)$ in Definition 2.

We first prove that
$$\varphi_{p-1, q-1}\preccurlyeq \varphi_{p, q}\qquad \text{and}
\qquad \varphi_{p, q-1}\preccurlyeq \varphi_{p-1, q-1} .\leqno (2.1)$$
Let $ \Delta $ be the diagonal of $X\times X$ and $\pi_{1},\pi_{2}$ the two
direction projections from $X\times X$ to $X$.
We define a new bundle on $X\times X$:
$$F=2^{p-1}\pi_{1}^{*}G+2^{p-1}\pi_{2}^{*}G+2^{q-1}\pi_{1}^{*}
L+2^{q-1}\pi_{2}^{*}L$$ 
with the singular metric
$$2^{q-1}\pi_{1}^{*}(\varphi)+2^{q-1}\pi_{2}^{*}(\varphi).$$
Since $2^{p-1}G-K_{X}$ is enough ample, we can apply the Ohsawa-Takegoshi extension
theorem from $\Delta$ to $X\times X$ for the line bundle $F$.
Thus the following map is surjective:
$$(H^{0}(X, (2^{p-1}G+2^{q-1}L)\otimes\mathcal{I}(2^{q-1}\varphi)))^{2} \leqno (2.2)$$
$$\rightarrow H^{0}(X, 2^{p}G+2^{q}L\otimes\mathcal{I}(2^{q}\varphi)).$$
Let $ \{f_{p-1, q-1 , i} \}_{i=1}^{N} $ be an orthonormal basis of
$$H^{0}(X, 2^{p-1}G+2^{q-1}L\otimes\mathcal{I}(2^{q-1}\varphi))$$
with respect to the singular metric $h_{G}^{2^{p-1}}\cdot h_{0}^{2^{q-1}}\cdot e^{- 2^{q-1}\varphi}$.
For any 
$$g\in H^{0}(X, 2^{p}G+2^{q}L\otimes\mathcal{I}(2^{q}\varphi)) ,$$
applying the effective version of Ohsawa-Takegoshi extension theorem to the morphism $(2.2)$,
we obtain the 
existence of constants $ \{c_{i, j} \} $ such that
$$g(z)=(\sum\limits_{i,j} c_{i,j} f_{p-1, q-1 ,i}(z) f_{p-1, q-1 ,j} (w) )\arrowvert _{z=w}$$
with
$$\sum\limits_{i,j} \lvert c_{i,j}\rvert^{2}\leq C_{1} \|g\|^{2}$$
where $ C_{1} $ depends only on $ X $ and $\|g\| $ is the $L^{2}$-norm with
respect to the singular metric $h_{G}^{2^{p}}\cdot h_{0}^{2^{q}} e^{- 2^{q}\varphi}$.
By the Cauchy inequality, we have
$$\lvert g(z)\rvert^{2}_{h_{G}^{2^{p}}\cdot h_{0}^{2^{q}}}\leq (\sum\limits_{i,j} \lvert
c_{i,j}\rvert^{2})(\sum\limits_{i,j} | f_{p-1, q-1 , i}(z) f_{p-1, q-1, j} (z)|^{2}_{h_{G}^{2^{p}}\cdot h_{0}^{2^{q}}}
)$$
$$\leq C_{1}\|g\|^{2} (\sum\limits_{i} | f_{p-1, q-1 , i}(z) |^{2}_{h_{G}^{2^{p-1}}\cdot h_{0}^{2^{q-1}}})^{2}.$$
Thus for $g$ with norm $\| g \|=1$, we have
$$\frac{1}{2^{q}}\ln \lvert g(z)\rvert^{2}_{h_{G}^{2^{p}}\cdot h_{0}^{2^{q}}}\leq \frac{\ln
C_{1}}{2^{q}}+\frac{1}{2^{q-1}} \ln(\sum\limits_{i} | f_{p-1, q-1 , i}(z) |^{2}_{h_{G}^{2^{p-1}}\cdot h_{0}^{2^{q-1}}})$$
$$=\frac{\ln C_{1}}{2^{q}}+\varphi_{p-1, q-1}(z).$$
By the extremal property of Bergman kernel, we obtain finally
$$\varphi_{p-1, q-1}\preccurlyeq \varphi_{p, q} .$$
The second inequality in $(2.1)$ is obvious by observing that $G$ is very ample.  
Thanks to the construction of $p_{m}, q_{m}$, 
$(2.1)$ implies that $\varphi_{m-1}\preccurlyeq \varphi_{m}$.
Therefore $\varphi_{m}$ is a quasi-equisingular approximation to $\varphi$ 
for the current $\frac{i}{2\pi}\Theta_{h_{0}}(L)+dd^{c}\varphi$.

We check the two properties listed in the proposition.
Property $(i)$ comes directly from the construction of $\varphi_{m}$. 
Since $\nd (L, \varphi)= n$ and  $\varphi_{m}$ is an quasi-equisingular approximation, 
the definition of the numerical dimension implies property $(ii)$.

\end{proof}

The rest of this section is devoted to prove Proposition 1.1.
The strategy is as follows.                                            
Thanks to property $(ii)$ of the approximating sequence $\{\varphi_{m}\}$, 
we can construct a new metric on $L$ with strictly positive curvatures
and the new metric is much singular than $\varphi$ in an asymptotic way.
Proposition 1.1 follows by the standard estimate for this new metric.
Before giving the construction of this new metric, 
we first prove two preparatory propositions.

\begin{prop}
Let $\varphi_{m}$ be the quasi-psh function constructed in Proposition 3.2. 
Then there exists another quasi-psh function $\widetilde{ \varphi}_{m} $
such that 
$$\sup\limits_{x\in X} \widetilde{ \varphi}_{m}(x)=0, \text{  } 
\frac{i}{2\pi}\Theta_{\widetilde{ \varphi}_{m}}(L)\geq \frac{\delta}{2}\cdot\omega
\text{  and  } \varphi_{m}\preccurlyeq\widetilde{
\varphi}_{m} ,$$ 
where $ \delta $ is a strictly positive number independent of $m$. 
\end{prop}

\begin{proof}

Let $\pi : X_{m}\rightarrow X$ be a log resolution of $\varphi_{m}$.
We can hence assume that
$$\frac{i}{2\pi}\Theta_{\varphi_{m}\circ\pi}(\pi^{*}(L))=[E]+C^{\infty} ,$$
where $ [E] $ is a normal crossing divisor.
Keeping the notations in Proposition 3.2,
we can suppose $A$ to be a $\mathbb{Q}$-ample line bundle such that
$c_{1}(A)=\omega$, and
$\epsilon$ to be a positive rational number in property $(ii)$ of Proposition
3.2, i.e.
$$\int_{X}(\frac{i}{2\pi}\Theta_{\varphi_{m}}(L)+\epsilon\omega)_{\ac} ^{n}> C .$$
Thanks to Proposition 3.2,
$$(\frac{i}{2\pi}\Theta_{\varphi_{m}\circ\pi}(\pi^{*}(L))+\epsilon\pi^{*}\omega)_{\ac}$$
is a $\mathbb{Q}$-nef class for $m$ large enough. 
We can thus choose a $\mathbb{Q}$-nef line bundle $F_{m}$ on $X_{m}$
such that
$$c_{1}
(F_{m})=(\frac{i}{2\pi}\Theta_{\varphi_{m}\circ\pi}(\pi^{*}(L))+\epsilon\pi^{*}\omega)_{\ac}.$$

Thanks to property $(ii)$ of Proposition $3.2$ 
and the cohomological fact that $F_{m}^{n-1}\cdot \pi^{*}A$ are uniformly
bounded for all $m$, 
we can thus choose a constant $\delta>0$ independent of $\epsilon$ and $m$, 
such that 
$$F_{m}^{n}> n \delta F_{m}^{n-1}\cdot \pi^{*}A . \leqno (3.1)$$
Using the holomorphic Morse inequality (cf. \cite{Dem} or \cite{Tra})
for the $\mathbb{Q}$-bundle $F_{m}-\delta\cdot\pi^{*}(A)$ on $X_{m}$, we have
$$h^{0}(X_{m}, kF_{m}-k\delta \cdot\pi^{*}A)\geq C \frac{k^{n}}{n !}
(F_{m}^{n}-n \delta F_{m}^{n-1}\cdot \pi^{*}A)+O(k^{n-1}).\leqno (3.2)$$
Combining $(3.1)$ and $(3.2)$, 
we get that
$F_{m}-\delta\pi^{*}\omega$ is pseudo-effective. 
In particular, if we take $\epsilon\leq\frac{\delta}{2}$,
the pseudo-effectiveness of $F_{m}-\delta\pi^{*}\omega$
implies the existence of a quasi-psh function $ \psi_{m} $ on $X_{m}$ such that
$$\frac{i}{2\pi}\Theta_{\varphi_{m}\circ\pi}(\pi^{*}(L))+dd^{c}\psi_{m}\geq
\frac{\delta}{2} \pi^{*}\omega . \leqno (3.3)$$
Choosing a constant $C_{1}$ such that
$$\sup\limits_{x\in X_{m}}(\varphi_{m}\circ\pi+\psi_{m}+C_{1})(x)=0 ,$$
$(3.3)$ implies that $\varphi_{m}\circ\pi(x)+\psi_{m}(x)+C_{1}$
induces a quasi-psh function on $X$ satisfying all the requirements
in the proposition.
We denote it by $\widetilde{\varphi}_{m}$.
\end{proof}

\begin{remark}
The essential point here is the holomorphic Morse inequality. 
Therefore $\frac{i}{2\pi}\Theta_{\varphi_{m}\circ\pi}$ is required to be rational. 
By using other techniques, 
Demailly and Paun proved in \cite{DP 04} the same results for the case of a real class.
\end{remark}

Thanks to Proposition $3.3$, 
we are going to construct a singular metric on $L$ which is a kind of limit of $\widetilde{ \varphi}_{m}$.
We first recall the notion of upper semicontinuous regularization.
Let $\Omega\subset \mathbb{R}^{n}$ and 
let $(u_{\alpha})_{\alpha\in I}$ be a family of upper semicontinuous fuctions 
$\Omega\rightarrow [-\infty, +\infty[$. 
Assume that $(u_{\alpha})$ is locally uniformly bounded from above.
Since the upper envelope
$$u=\sup_{\alpha\in I} u_{\alpha}$$
need not be upper semicontinuous, 
we consider its upper semicontinuous regularization:
$$u^{*}(z)=\lim_{\epsilon\rightarrow 0}\sup_{B(z,\epsilon)} u.$$
We denote this upper semicontinuous regularization $\widetilde{\sup\limits_{\alpha}}(u_{\alpha})$.
It is easy to proof that if $\{ u_{\alpha}\}_{\alpha\in I}$ 
are psh fonctions with locally uniformly bounded from above, 
then $\widetilde{\sup\limits_{\alpha}}(u_{\alpha})$ is also a psh function
(cf. \cite{Dem} for details).

We need the following lemma.
\begin{lemme}
Let $\varphi$ be a quasi-psh fonction with normal crossing singularities
and let $ \{\psi_{i} \} $ be quasi-psh functions
such that 
$$\sup\limits_{z\in X}
\psi_{i}(z)\leq 0\qquad\text{and}\qquad  dd^{c}\psi_{i}\geq -C\omega$$
for some uniform constant $C$ independent of $i$.
If $\varphi\preccurlyeq \psi_{i} $ for all $ i $, 
then 
$$ \varphi\preccurlyeq\widetilde{\sup\limits_{i}} (\psi_{i}) .$$
\end{lemme}

\begin{proof}

Since $\varphi$ has normal crossing singularities and $\varphi$ is less singular than $\varphi_{i}$, 
$ \psi_{i} - \varphi $ are quasi-psh functions 
and
$$dd^{c}(\psi_{i}-\varphi) \geq -C_{1}\omega \leqno (4.1)$$
for some uniform constant $C_{1}$ independent of $i$.

On the other hand, 
since $\sup\limits_{z\in X} \psi_{i}(z)\leq 0$ and $dd^{c}\varphi_{i}\geq
-C\omega$ for a uniform $C$, 
by the standard potential theory, there exists a constant $ M $ such that
$$\int_{X}\psi_{i}\leq M \qquad\text{for all}\hspace{5 pt} i.$$
Therefore 
$$\int_{X}(\psi_{i}-\varphi)\leq M'\leqno (4.2)$$
for a uniform constant $M'$.

$(4.1)$ and $(4.2)$ imply the existence of a uniform constant $ C_{2}$ such
that 
$$\sup\limits_{z\in X} (\psi_{i}(z)-\varphi(z))\leq C_{2} \qquad\text{for all}\hspace{5 pt} i.$$
Then $ \varphi\preccurlyeq\widetilde{\sup\limits_{i}} (\psi_{i}).$
The lemma is proved.
\end{proof}

The following metric will be useful in our context.
Using the same termology as above, we have

\begin{prop}
Let $\widetilde{\varphi}(z)=\lim\limits_{m\rightarrow
\infty}\widetilde{\sup\limits_{s\geq 0}} ((\widetilde{\varphi}_{m+s}(z)))$.
Then the new metric $\widetilde{\varphi}$ satisfies:
$$\frac{i}{2\pi}\Theta_{\widetilde{\varphi}}(L)\geq \frac{\delta}{2}\omega\qquad\text{and}\qquad
\varphi_{m}\preccurlyeq \widetilde{\varphi}\leqno (5.0)$$
for any $m\geq 1$.
\end{prop}

\begin{proof}

By Proposition $3.3$, we have
$$\frac{i}{2\pi}\Theta_{\widetilde{\varphi}_{m}}(L)\geq \frac{\delta}{2}\omega \qquad\text{for } m\geq 1.$$
Then $\frac{i}{2\pi}\Theta_{\widetilde{\varphi}}(L)\geq \frac{\delta}{2}\omega .$
To check $\varphi_{m}\preccurlyeq \widetilde{\varphi}$,
it is enough to show that
$$\varphi_{m}\preccurlyeq \widetilde{\sup\limits_{s\geq 0}}
(\widetilde{\varphi}_{m+s})\leqno
(5.1)$$
by observing $\widetilde{\varphi} \leq \widetilde{\sup\limits_{s\geq 0}}
(\widetilde{\varphi}_{m+s})$.
Combining Proposition $3.3$ and Proposition $3.2$, we have 
$$\varphi_{m} \preccurlyeq \varphi_{m+s}  
\preccurlyeq \widetilde{\varphi}_{m+s} \qquad\text{for any}\hspace{5 pt} m , s. \leqno (5.2)$$
Let $ \pi $ be a log resolution of $ \varphi_{m}$. 
Then
$$\varphi_{m}\circ\pi \preccurlyeq \varphi_{m+s}\circ\pi  \preccurlyeq
\widetilde{\varphi}_{m+s}\circ\pi.\leqno(5.3)$$
Thanks to Lemma $3.4$, $(5.3)$ implies that 
$$\varphi_{m}\circ\pi\preccurlyeq \widetilde{\sup\limits_{s\geq 0}}
(\widetilde{\varphi}_{m+s}\circ\pi).$$
By passing to $ \pi_{*} $, $(5.1)$ is proved.
\end{proof}

Using the new metric $\widetilde{\varphi}$, we can give an asymptotic estimate.

\begin{prop}
Let $(L,\varphi)$ be a pseudo-effective line bundle on a projective variety $X$
of dimension $n$ with $\nd(L,\varphi)=n$.
Then
$$\varliminf\limits_{m\to \infty}\frac{h^{0}(X,mL\otimes\mathcal{I}(m\varphi))}{m^{n}}> 0 .$$
\end{prop}

\begin{proof}

Lemma $3.1$ implies that for any $m\in \mathbb{N}$,
$$h^{0}(X,mL\otimes\mathcal{I}(m\varphi))\geq
h^{0}(X,mL\otimes\mathcal{I}(\frac{m\cdot
2^{q_{k}}}{2^{q_{k}}-m}\varphi_{k})).\leqno(6.1)$$
Let $\widetilde{\varphi}$ be the metric constructed in Proposition $3.5$. 
Then Proposition 3.5 implies that
$$h^{0}(X,mL\otimes\mathcal{I}(\frac{m\cdot
2^{q_{k}}}{2^{q_{k}}-m}\varphi_{k}))\geq
h^{0}(X,mL\otimes\mathcal{I}(\frac{m\cdot
2^{q_{k}}}{2^{q_{k}}-m}\widetilde{\varphi})).\leqno(6.2)$$
for any $ k,m$.
Combining $(6.1)$ and $(6.2)$, we get
$$h^{0}(X,mL\otimes\mathcal{I}(m\varphi))\geq
h^{0}(X,mL\otimes\mathcal{I}(\frac{m\cdot
2^{q_{k}}}{2^{q_{k}}-m}\widetilde{\varphi})).\leqno (6.3)$$
Since $(6.3)$ is true for all $m$ and $k$, if we take $k$ such that $2^{q_{k}}\gg m$, 
by applying $\frac{i}{2\pi}\Theta_{\widetilde{\varphi}}(L)> 0$ proved in Proposition $3.5$,
$(6.3)$ implies that
$$\varliminf\limits_{m\to \infty}\frac{h^{0}(X,mL\otimes\mathcal{I}(m\varphi))}{m^{n}}> 0 .$$
\end{proof}

\section{A numerical criterion}

Until now, we have two kind of numerical dimension of pseudo-effective line bundles:
$\nu_{\num}(L,\varphi)$ in Defintion 1 and more analytically 
$\nd (L,\varphi)$ in Definition 4.
We prove in this section that $\nu_{\num}(L,\varphi)=\nd (L,\varphi)$
when $X$ is projective.
Before giving the proof, 
we first list some properties of multiplier ideal sheaves which will be useful in our context.
The essential tool here is the Ohsawa-Takegoshi extension theorem.

\begin{lemme}
Let $ (L, \varphi) $ be a pseudo-effective line bundle on a projective variety $X$ of dimension $ n $ 
and let $\{\varphi_{k}\}$ be a quasi-equisingular approximation to $\varphi$.
Let $s_{1}$ be a positive number such that 
$$\mathcal{I}((1+\epsilon')\varphi)=\mathcal{I}_{+}(\varphi)
\qquad\text{for any}\hspace{5 pt} 0 <\epsilon'\leq s_{1} \leqno (*)$$
Assume that $ A $ is a very ample line bundle 
and $ S $ is the zero divisor of a general global section of $H^{0}(X, A) $. 
We have the following properties:

$(i)$ The restrictions
$$\mathcal{I}(m(1+\delta)\varphi_{k})\rightarrow \mathcal{I}(S, m(1+\delta)\varphi_{k}\mid_{S}) \leqno (1.1)$$
$$\mathcal{I}(m(1+\delta)\varphi)\rightarrow \mathcal{I}(S, m(1+\delta)\varphi\mid_{S}) \leqno (1.2)$$
are well defined for all $ m \in \mathbb{N} $, $ \delta \geq 0$,
where $\varphi\mid_{S}$ denotes the restriction of $\varphi$ on $S$
and $\mathcal{I}(S, \varphi\mid_{S})$ is the multiplier ideal sheaf associated to $\varphi\mid_{S}$ on $S$.
\footnote{$\varphi\mid_{S}$ is also psh if it is well defined.}
Moreover we have
$$\mathcal{I}(S,(1+\epsilon')\varphi\mid_{S})=\mathcal{I}(S, (1+s_{1})\varphi\mid_{S})
\qquad\text{for any}\hspace{5 pt} 0 <\epsilon'\leq s_{1} .$$

$(ii)$ $\{\varphi_{k}\mid_{S}\}$ is also a quasi-equisingular approximation to $\varphi\mid_{S}$.

$(iii)$ If the restrictions are well defined, we have an exact sequence:
$$0\rightarrow \mathcal{I}_{+}(\varphi)\otimes \mathcal{O}(-S)\rightarrow
\adj_{S}^{\epsilon}(\varphi)\rightarrow \mathcal{I}_{+}(S,
\varphi\mid_{S})\rightarrow 0$$
for any $ 0 <\epsilon\leq s_{1} $, where
$$\adj_{S}^{\epsilon}(\varphi)_{x}=\{ f\in \mathcal{O}_{x}, \int_{U_{x}} \frac{\mid
f\mid^{2}}{\mid s\mid^{2(1-\frac{\epsilon}{2})}}e^{-2(1+\epsilon)\varphi}<
+\infty\} $$
and 
$\mathcal{I}_{+}(S, \varphi\mid_{S})$ is multiplier ideal sheaf $\mathcal{I}_{+}(\varphi|_{S})$ on $S$.

$(iv)$ $\adj_{S}^{\epsilon}(\varphi)=\mathcal{I}_{+}(\varphi)$ for any $ 0 <\epsilon\leq s_{1} $.

\end{lemme}

\begin{proof}

$(i)$ By the Notherian property, except for countably many $s\in \mathbb{R}^{+}$, 
we have
$$\mathcal{I}(s\varphi)=\mathcal{I}((s+\delta)\varphi)$$
for $\delta>0$ small enough. Therefore there exists a countable set $I\subset \mathbb{R}^{+}$ 
such that for any $ t\in\mathbb{R}^{+}$, 
we can find an $\alpha\in I$ such that
$$\mathcal{I}(\alpha\varphi)=\mathcal{I}(t\varphi).\leqno (1.3)$$
Since $I$ is a countable set, 
we can take $S$ to be very general in such a way that the restrictions
$$\mathcal{I}(\alpha\varphi)\rightarrow \mathcal{I}(S, \alpha\varphi\mid_{S}) $$
are well defined for all $\alpha\in I$. 
Then $(1.3)$ implies that $(1.1)$ and $(1.2)$ are well defined for all $ m, \delta\geq 0$.

For the second part of $(i)$, if the restrictions
$$\mathcal{I}(\varphi)\rightarrow \mathcal{I}(S, \varphi\mid_{S})$$
$$\mathcal{I}((1+\delta)\varphi)\rightarrow \mathcal{I}(S, (1+\delta)\varphi\mid_{S})$$
are well defined, then the Ohsawa-Takegoshi extension theorem and $(*)$ imply immediately that
$$\mathcal{I}(S,(1+\epsilon')\varphi\mid_{S})=\mathcal{I}(S, (1+s_{1})\varphi\mid_{S}).$$

$(ii)$ Since $\{\varphi_{k}\}$ is a quasi-equisingular approximation to $\varphi$, 
we have
$$\mathcal{I}(m(1+\delta)\varphi_{k})\subset \mathcal{I}(m\varphi)
\qquad\text{for any}\hspace{5 pt} k\geq k_{0}(\delta, m) .$$
Using $(1.1),(1.2)$ and the Ohsawa-Takegoshi extension theorem, we have
$$\mathcal{I}(S, m(1+\delta)\varphi_{k}\mid_{S})\subset \mathcal{I}(S, m\varphi\mid_{S})
\qquad\text{for any}\hspace{5 pt} k\geq k_{0}(\delta, m) .$$
Therefore $\varphi_{k} \mid_{S} $ is a quasi-equisingular approximation to $\varphi\mid_{S}$.
\vspace{5 pt}

$(iii)$ First of all, the Ohsawa-Takegoshi extension theorem implies the surjectivity of the sequence.
We need only to prove the exactness of the middle term,
i.e., for any $f\in \mathcal{O}_{x}$ satisfying the conditions
$$\frac{f}{s}\in \mathcal{O}_{x} 
\qquad\text{and}\qquad
\int_{U_{x}} \frac{\mid f\mid^{2}}{\mid
s\mid^{2(1-\frac{\epsilon}{2})}}e^{-2(1+\epsilon)\varphi}< +\infty ,\leqno
(1.4)$$
we should prove the existence of some $ \epsilon'>0$ such that
$$\int\frac{\mid f\mid^{2}}{\mid s\mid^{2}}e^{-2(1+\epsilon')\varphi}< +\infty ,\leqno (1.5)$$
where $s$ is a local function defining $S$.
In fact, if $\frac{f}{s}\in \mathcal{O}_{x}$, then
$$\int_{U_{x}}\frac{\mid f\mid^{2}}{\mid s\mid^{4-\delta}}< +\infty
\qquad\text{for any } \delta>0 .\leqno (1.6)$$
If we take $\epsilon'=\frac{\epsilon}{4}$ in the left side of $(1.5)$, then
$$\int_{U_{x}}\frac{\mid f\mid^{2}}{\mid s\mid^{2}}e^{-2(1+\frac{\epsilon}{4})\varphi}\leqno (1.7)$$
$$\leq (\int_{U_{x}} \frac{\mid f\mid^{2}}{\mid
s\mid^{2(1-\frac{\epsilon}{2})}}e^{-2(1+\epsilon)\varphi})^{\frac{1+\frac{
\epsilon}{4}}{1+\epsilon}}(\int_{U_{x}} \frac{\mid f\mid^{2}}{\mid
s\mid^{\alpha}})^{\frac{\frac{3\epsilon}{4}}{1+\epsilon}}$$
by H\"{o}lder's inequality, where
$$\alpha=(2-2(1-\frac{\epsilon}{2})\frac{1+\frac{\epsilon}{4}}{1+\epsilon})\cdot
(1+\epsilon)\cdot\frac{4}{3\epsilon}=\frac{10\epsilon+\epsilon^{2}}{3\epsilon}<
4 .$$
Thanks to $(1.4)$ and $(1.6)$, the second line of $(1.7)$ is finite. 
Thus $(1.5)$ is proved.
\vspace{5 pt}

$(iv)$ By the definition of $\mathcal{I}_{+}(\varphi)$, we have an obvious inclusion
$$\adj_{S}^{\epsilon}(\varphi)\subset\mathcal{I}_{+}(\varphi) .$$
In order to prove the equality, it is enough to show that
for any $ f\in \mathcal{I}((1+\epsilon)\varphi)_{x}$, we have 
$$\int_{U_{x}} \frac{\mid f\mid^{2}}{\mid
s\mid^{2(1-\frac{\epsilon}{2})}}e^{-2(1+\epsilon)\varphi}dV < +\infty, \leqno
(1.8)$$
where $s$ is a general global section of $H^{0}(X,A)$ independent of the choice of $f$ and $x$.

$(1.8)$ comes from the Fubini theorem. 
In fact, let $\{s_{0},\cdots,s_{N}\}$ be a basis of $H^{0}(X,A)$. 
Then 
$$\sum_{i=0}^{N} |s_{i}(x)|^{2}\neq 0\qquad\text{for any}\hspace{5 pt} x\in X.$$
Taking $\{\tau_{0},\cdots ,\tau_{N}\}\in \mathbb{C}^{N+1}$, we have
$$\int_{\sum\limits_{i=0}^{N} |\tau_{i}|^{2}=1} d\tau 
\int_{U_{x}}  \frac{\mid f\mid^{2}}{\mid \sum\limits_{i=0}^{N}
\tau_{i}s_{i}\mid^{2(1-\frac{\epsilon}{2})}}e^{-2(1+\epsilon)\varphi} d V\leqno
(1.9)$$
$$=\int_{U_{x}}  \frac{\mid f\mid^{2}}{\mid \sum\limits_{i=0}^{N}
|s_{i}(x)|^{2}\mid^{(1-\frac{\epsilon}{2})}}e^{-2(1+\epsilon)\varphi} d V 
\int_{\sum\limits_{i=0}^{N} |\tau_{i}|^{2}=1}\frac{1}{(\sum\limits_{i=0}^{N}
\tau_{i}\frac{s_{i}}{\sum\limits_{i=0}^{N}
|s_{i}(x)|^{2}})^{2(1-\frac{\epsilon}{2})}} d \tau$$
$$=\int_{U_{x}}  \frac{\mid f\mid^{2}}{\mid \sum\limits_{i=0}^{N}
|s_{i}(x)|^{2}\mid^{(1-\frac{\epsilon}{2})}}e^{-2(1+\epsilon)\varphi} d V
\int_{\sum\limits_{i=0}^{N}
|\tau_{i}|^{2}=1}\frac{1}{|\tau_{0}|^{2(1-\frac{\epsilon}{2})}}d \tau <+\infty$$
For any $f\in \mathcal{I}((1+\epsilon)\varphi)_{x}$ fixed,
applying the Fubini theorem to $(1.9)$, we obtain 
$$\int_{U_{x}}  \frac{\mid f\mid^{2}}{\mid
s\mid^{2(1-\frac{\epsilon}{2})}}e^{-2(1+\epsilon)\varphi}< +\infty \leqno (1.10)$$
for a general element $s\in H^{0}(X,A)$.
Observing that $\mathcal{I}((1+\epsilon)\varphi)$ is finitely generated on $X$, 
we can thus choose a general section $s$ such that
$(1.10)$ is true for any $f\in \mathcal{I}((1+\epsilon)\varphi)$.
The equality $\adj_{S}^{\epsilon}(\varphi)=\mathcal{I}_{+}(\varphi)$ is proved.
\end{proof}

The next proposition confirms that our definition of numerical dimension coincides with Tsuji's definition.

\begin{prop}
If $(L,\varphi)$ is a pseudo-effective on a projective variety $X$ of dimension $n$, then 
$$ \nu_{\num} (L, \varphi) =\nd (L,\varphi) .$$
\end{prop}

\begin{proof}

We first prove $\nu_{\num}(L, \varphi)\geq \nd (L,\varphi)$ by induction on dimension.
If $ \nd (L, \varphi) = n $, the inequality comes from Proposition 3.6.
Assume now $ \nd (L, \varphi) <n $. Let $ A $ be a general hypersurface given by a very ample line bundle
and let $\{\varphi_{k}\}$ be a quasi-equisingular approximation to $\varphi$.
Lemma 4.1 implies that
$\varphi_{k}\mid_{A}$ is also a quasi-equisingular approximation to $\varphi\mid_{A}$.
Since $ A $ is a general section and $ \nd (L, \varphi) <n $, we have
$$\lim\limits_{k\to \infty} \int_{A}((\frac{i}{2\pi}\Theta_{\varphi_{k}})_{\ac})^{s}\wedge \omega^{n-s-1}>0$$
where $s=\nd(L,\varphi)$.
Thus Definition 2 implies that
$$\nd(A,L,\varphi\mid_{A})\geq s=\nd(L,\varphi) . \leqno (2.1)$$
Notice moreover that the definition of $ \nu_{\num} $ implies
$$\nu_{\num}( L, \varphi)\geq \nu_{\num}(A, L, \varphi\mid_{A}). \leqno (2.2)$$
Combining $(2.1)$ and $(2.2)$,
the equality $\nu_{\num}(L, \varphi)\geq \nd (L,\varphi)$ is concluded 
by induction on dimension.
\vspace{5 pt}

We now prove $\nu_{\num}(L,\varphi)\leq \nd(L,\varphi) .$
Assume that $\nu_{\num} (L, \varphi)=k$. 
By definition, there exists a subvariety $ V $ of dimension $ k $ such that 
$$\varlimsup\limits_{m\to \infty}\frac{h^{0}(V, mL\otimes\mathcal{I}(m\varphi))}{m^{k}}> 0 .$$
Let $\pi:\widetilde{X}\rightarrow X$ be the desingularization of the ideal sheaf of $V$ in $X$,
and let $\widetilde{V}$ be the strict transform of $V$. 

By the Ohsawa-Takegoshi extension theorem, 
there exists a very ample line bundle $A$ on $\widetilde{X}$ 
such that the following restrictions are surjective
$$H^{0}(\widetilde{X}, A+m\pi^{*}(L)\otimes \mathcal{I}(m\varphi\circ\pi))
\rightarrow H^{0}(\widetilde{V}, A+m\pi^{*}(L)\otimes \mathcal{I}(m\varphi\circ\pi))\leqno (2.3)$$
for all $m>0$.
Fix a smooth metric $ h_{0} $ on $ L $ and 
let $ \{e_{i, m} \} $ be an orthonormal basis of 
$H^{0}(\widetilde{X}, A+m\pi^{*}(L)\otimes \mathcal{I}(m\varphi\circ\pi))$ 
with respect to $h_{A}\cdot h_{0}^{m} e^{-m\varphi\circ\pi}$.
It is well known that we can take a smooth function $\psi$
\footnote{we can take a smooth metric on the exceptional divisor $\mathcal{O}(-E)$, 
and the canonical section gives this function. See \cite{Bou} for details.}
independent of $m$ on $\widetilde{X}$ such that
$$\frac{1}{m} i\partial\overline{\partial} (\psi+ \ln\sum\limits_{i}
|e_{i,m}|_{h_{A}, h_{0}^{m}} ^{2} )\geq -C\pi^{*}\omega_{X}.$$
Then $\frac{1}{m} (\psi+ \ln\sum\limits_{i} |e_{i,m}|_{h_{A}, h_{0}^{m}} ^{2} )$
induces a quasi-psh fucntion on $X$
and we denote it by $\varphi_{m}$.
We claim that
$$\lim\limits_{m\to
\infty}\int_{X}(\frac{i}{2\pi}\Theta_{\varphi_{m}}(L)+\frac{1}{m}\omega)_{\ac} ^{k}\wedge
\omega^{n-k}>0 .\leqno (2.4)$$

We postpone the proof of $(2.4)$ in Lemma 4.3 and conclude first the proof of Proposition 4.2.
By Lemma 3.1, $\{\varphi_{m}\circ\pi\}$ is a quasi-equisingular approximation to $\varphi\circ\pi$.
Thanks to the formula 
$$\pi_{*}(K_{\widetilde{X}/ X}\otimes\pi^{*}\mathcal{I}(\varphi))=\mathcal{I}(\varphi),$$
$\{\varphi_{m}\}$ is a quasi-equisingular approximation to $\varphi$.
Therefore $(2.3)$ implies $\nd(L,\varphi)\geq k$. 
Since $\nu_{\num}(L,\varphi)=k$ by assumption, 
we conclude that $\nd(L,\varphi)\geq\nu_{\num}(L,\varphi)$.
The proposition is proved.
\end{proof}

\begin{remark}
From the proof, it is easy to conclude that 
if $ S_{1}, S_{2 },..., S_{k} $ are divisors of general global sections of a very ample line bundle, then
$$\nd(S_{1}\cap S_{2}\cap\cdots \cap S_{k}, L,\varphi)=\max (\nd(L, \varphi), n-k) .\leqno (*)$$
In fact, if $\nd(L, \varphi)\leq n-k$, by the same argument as above, 
$\varphi_{m}\mid_{S_{1}\cap S_{2}... \cap S_{k}}$ 
is also a quasi-equisingular approximation to $\varphi\mid_{S_{1}\cap S_{2}\cap\cdots
\cap S_{k}}$. 
Then $(*)$ is proved by simple calculation.
\end{remark}

We now prove Lemma 4.3 promised in the above proposition.
\begin{lemme}
We have 
$$\lim\limits_{m\to
\infty}\int_{X}(\frac{i}{2\pi}\Theta_{\varphi_{m}}(L)+\frac{1}{m}\omega)_{\ac} ^{k}\wedge
\omega^{n-k}>0.\leqno (3.1)$$
\end{lemme}
\begin{proof}

Let $ \{e_{i, m }' \} $ be an orthonormal basis of 
$H^{0}(\widetilde{V}, A+m\pi^{*}(L)\otimes \mathcal{I}(m\varphi\circ\pi))$  
with respect to $ h_{A}\cdot h_{0}^{m} e^{-m\varphi\circ\pi} $ and
let $\varphi_{m}'=\frac{1}{m} (\psi+ \ln\sum\limits_{i} |e_{i,m}'|_{h_{A}, h_{0}^{m}} ^{2} )$ 
which is a psh function on $\widetilde{V}$.
We prove $(3.1)$ in two steps.
\vspace{5 pt}

\textbf{Step 1:}
We first prove that 
$$\int_{X}(\frac{i}{2\pi}\Theta_{\varphi_{m}}(L)+\frac{1}{m}\omega)_{\ac} ^{k}\wedge
\omega^{n-k}\geq
C_{1}\int_{\widetilde{V}}(\frac{i}{2\pi}\Theta_{\varphi_{m}\circ\pi}(L)+\frac{1}{m}\omega)_{\ac}
^{k}=C_{1}\int_{\widetilde{V}}(\frac{i}{2\pi}\Theta_{\varphi_{m}'}(L)+\frac{1}{m}\omega)_{\ac}
^{k}$$
for some constant $ C_{1}> 0 $.
Noting that we abuse a little bit the notation here. 
We denote $L$ also for the pull back of $L$ on $\widetilde{X}$.

In fact, there exists a $ C_{1} $ such that $[V]\leqslant C_{1}\omega^{n-k}$ in the sense of cohomology class,
and $(\frac{i}{2\pi}\Theta_{\varphi_{m}}(L)+\frac{1}{m}\omega)_{\ac}$ is the push forward of a nef class.
Therefore
$$\int_{X}(\frac{i}{2\pi}\Theta_{\varphi_{m}}(L)+\frac{1}{m}\omega)_{\ac} ^{k}\wedge \omega^{n-k}
\geq C_{1}\int_{V \setminus V_{\sing}}(\frac{i}{2\pi}\Theta_{\varphi_{m}}(L)+\frac{1}{m}\omega)_{\ac} ^{k}.$$
Since $\widetilde{V}$ is the strict transformation of $V$, we have
$$\int_{V}(\frac{i}{2\pi}\Theta_{\varphi_{m}}(L)+\frac{1}{m}\omega)_{\ac}
^{k}=\int_{\widetilde{V}}(\frac{i}{2\pi}\Theta_{\varphi_{m}\circ\pi}(L)+\frac{1}{m}\omega)_{\ac}
^{k} .$$
We obtain thus the first inequality.

We now prove the second equality. 
Since $\psi$ is a smooth function on $\widetilde{X}$, 
we have $(dd^{c}\psi)_{\ac}=dd^{c}\psi$.
Therefore $(dd^{c}\psi)_{\ac}$ has no contribution for the integral on $\widetilde{V}$
by a cohomological reason.
We can thus suppose for simplicity that
$$\frac{i}{2\pi}\Theta_{\varphi_{m}}(L)=\frac{i}{2\pi}\Theta_{h_{0}}(L)+\frac{1}{m}dd^{c}\ln \sum_{i}
|e_{i,m}|_{h_{A}, h_{0}^{m}} ^{2}$$
and
$$\frac{i}{2\pi}\Theta_{\varphi_{m}'}(L)=\frac{i}{2\pi}\Theta_{h_{0}}(L)+\frac{1}{m}dd^{c}\ln \sum_{i}
|e_{i,m}'|_{h_{A}, h_{0}^{m}} ^{2}.$$
Thanks to the surjectivity of $(2.3)$ in Proposition 4.2 and the extremal property of Bergman kernels,
$$\frac {\sum\limits_{i} |e_{i,m}|_{h_{A}, h_{0}^{m}} ^{2}} {\sum\limits_{i}
|e_{i,m}'|_{h_{A}, h_{0}^{m}} ^{2}}$$
is thus a smooth function and does not vanish on $\widetilde{V}$. 
Therefore $(\frac{i}{2\pi}\Theta_{\varphi_{m}\circ\pi})_{\ac}(L)\mid_{\widetilde{V}}$ 
and $(\frac{i}{2\pi}\Theta_{\varphi_{m}'})_{\ac}(L)\mid_{\widetilde{V}}$ are in the same cohomology class. 
The equality is proved.
\vspace{5 pt}

\textbf{Step 2:}
We prove in this step the existence of a uniform constant $C>0$ such that
$$\int_{\widetilde{V}}(\frac{i}{2\pi}\Theta_{\varphi_{m}'}(L))_{\ac} ^{k}\geq C \qquad\text{for } m\gg 1.$$

For any $ m $ fixed, since $\varphi_{m}'$ is less singular than $\varphi$, 
we have
$$h^{0}(\widetilde{V}, A+sL\otimes \mathcal{I}(s\varphi))\leq
h^{0}(\widetilde{V}, A+sL\otimes\mathcal{I}(s \varphi_{m}')) . \leqno (3.2)$$
Notice that by the bigness of $(V, L, \varphi) $ we have
$$\varlimsup\limits_{s\to \infty}\frac{h^{0}(\widetilde{V},
A+sL\otimes\mathcal{I}(s\varphi))}{s^{k}}\geq C_{0}$$
for some constant $C_{0}>0$.
Then $(3.2)$ implies
$$h^{0}(\widetilde{V}, A+sL\otimes\mathcal{I}(s\varphi_{m}'))
\geq h^{0}(\widetilde{V}, A+sL\otimes\mathcal{I}(s \varphi))\geq C_{0}
s^{k}\leqno (3.3)$$
for a sequence $s\rightarrow\infty$.

On the other hand, since $ \varphi_{m}'$ has analytic singularities, we have
$$h^{0}(\widetilde{V}, A+sL\otimes\mathcal{I}(s\varphi_{m}'))= 
C_{1} s^{k}\int_{\widetilde{V}}(\frac{i}{2\pi}\Theta_{\varphi_{m}'})_{\ac} ^{k} +O (s^{k-1})
\leqno (3.4)$$
for $m\gg 1$.
Combining $(3.3)$ and $(3.4)$, we get
$$\int_{\widetilde{V}}(\frac{i}{2\pi}\Theta_{\varphi_{m}'}(L))_{\ac} ^{k}\geq \frac{C_{0}}{C_{1}}
> 0 \qquad\text{for } m\gg 1 .$$
Step 2 is proved. 

Combining Steps 1 and 2, the lemma is proved.
\end{proof}

We now give a numerical criterion to calculate the numerical dimension.
\begin{prop}
Let $(L,\varphi)$ be a pseudo-effective line bundle on a projectvie variety $X$, 
and let $ A $ be a very ample line bundle. 
Then $ \nd (L, \varphi) = d $ if and only if
$$\lim\limits_{\epsilon\to 0}\frac{\ln(\varlimsup\limits_{m\rightarrow \infty}
\frac{h^{0}(X, m\epsilon A+mL\otimes \mathcal{I}(m\varphi))}{m^{n}})}{\ln
\epsilon}=n-d .$$
\end{prop}

\begin{proof}

First of all, the inclusion
$$H^{0}(X, m\epsilon A+mL\otimes \mathcal{I}(m\varphi))\supset H^{0}(X,
m\epsilon A+mL\otimes \mathcal{I}_{+}(m\varphi)) $$
$$\supset H^{0}(X, m\epsilon A+mL\otimes \mathcal{I}((m+1)\varphi)),$$
implies that $h^{0}(X, m\epsilon A+mL\otimes \mathcal{I}_{+}(m\varphi))$ has the same
asymptotic comportment as $h^{0}(X, m\epsilon A+mL\otimes
\mathcal{I}(m\varphi))$.
Since we have constructed the exact sequence for $\mathcal{I}_{+}$ in Lemma 4.1, 
we prefer to calculate $h^{0}(X, m\epsilon A+mL\otimes \mathcal{I}_{+}(m\varphi))$ in the following argument.

If $ \nd(L, \varphi) = n $, the proposition comes directly from Proposition 4.2.
Assume now that $ \nd(L, \varphi) = d < n$. Let $ \{Y_{i} \}_{i=1}^{n} $ be $n$ general
sections of $ H^{0} (X, A) $. 
By the remark of Proposition 4.2, 
there exists a uniform constant $C>0$ such that for all $m, \epsilon$,
$$h^{0}(Y_{1}\cap\cdots\cap Y_{n-d}, m\epsilon A+mL\otimes
\mathcal{I}_{+}(m\varphi))=C(\epsilon,m)m^{d}. \leqno (4.1)$$
and $ C (\epsilon, m) \geq C$.
We prove by induction on $s$ that
$$\frac{1}{m^{n-s}}h^{0}(Y_{1}\cap\cdots\cap Y_{s}, m\epsilon A+mL\otimes \mathcal{I}_{+}(m\varphi))\leqno (4.2)$$
$$=C(\epsilon, m)\epsilon^{n-s-d}\frac{1}{(n-d-s)!}+O(\epsilon^{n-s-d+1})+O(\frac{1}{m})$$
for $ 0 \leq s \leq n-d $, 
where the constant $ C (\epsilon, m) $ is the same constant in $(4.1)$, in
particular it does not depend on the codimension $ s $.

If $ s = n-d$, $(4.2)$ comes from the definition of $ C (\epsilon, m) $.
We suppose that $(4.2)$ is true for $s_{0}\leq s\leq n-d $. 
We now prove that $(4.2)$ is also true for $s=s_{0}-1$.

Let $ Y $ be the intersection of zero divisors of $ s_ {0}-1 $
general elements of $H^{0}(X, A)$, and let
$$e_{1}^{0,q}(\epsilon, m)=\binom{m\epsilon}{q}h^{0}(Y\cap Y_{1}\cap\cdots \cap
Y_{q}, m\epsilon A\otimes mL\otimes\mathcal{I}_{+}(m\varphi)).\leqno (4.3)$$
We claim the formula:
$$\frac{1}{m^{n-s_{0}+1}}h^{0}(Y, m\epsilon A+mL\otimes
\mathcal{I}_{+}(m\varphi))\leqno
(4.4)$$
$$=-\frac{1}{m^{n-s_{0}+1}}(\sum\limits_{q\geq
1}(-1)^{q}e_{1}^{0,q}(\epsilon,m))+O(\frac{1}{m}).$$

We postpone the proof of $(4.4)$ in Lemma 4.5 and conclude first the proof of
the formula $(4.2)$.
First of all, if $q> n-d-s_{0}+1$, we have by definition 
$$\lim\limits_{m\rightarrow
\infty}\frac{1}{m^{n-s_{0}+1}}e_{1}^{0,q}(\epsilon)=O(\epsilon^{q})\leq
O(\epsilon^{n-d-s_{0}+2}).\leqno (4.5)$$
Thus $(4.4)$ and the induction hypothesis of $(4.2)$ implies that
$$\frac{1}{m^{n-s_{0}+1}}h^{0}(Y, m\epsilon A+mL\otimes
\mathcal{I}_{+}(m\varphi))$$
$$=-(\sum\limits_{q=1}^{n-d-s_{0}+1}
(-1)^{q}\frac{\epsilon^{n-d-s_{0}+1}C(\epsilon,m)}{q!(n-q-s_{0}+1-d)!}
)+O(\epsilon^{n-d-s_{0}+2})+O(\frac{1}{m})$$
$$=-(\sum\limits_{q=1}^{n-d-s_{0}+1}
(-1)^{q}\frac{\epsilon^{n-d-s_{0}+1}C(\epsilon,m)}{(n-s_{0}+1-d)!}\binom{n-s_{0}
+1-d}{q})+O(\epsilon^{n-d-s_{0}+2})+O(\frac{1}{m})$$
$$=-\frac{\epsilon^{n-d-s_{0}+1}C(\epsilon,m)}{(n-s_{0}+1-d)!}(\sum\limits_{q=1}^{n-d-s_{0}+1}
(-1)^{q}\binom{n-s_{0}+1-d}{q})+O(\epsilon^{n-d-s_{0}+2})+O(\frac{1}{m})$$
$$=C(\epsilon,
m)\epsilon^{n-d-s_{0}+1}\frac{1}{(n-d-s_{0}+1)!}+O(\epsilon^{n-d-s_{0}+2}
)+O(\frac{1}{m}).$$
Therefore $(4.2)$ is proved for $s=s_{0}-1$. 

In particular, since $(4.1)$ implies that 
$$\lim\limits_{\epsilon\to 0}\varlimsup\limits_{m\rightarrow \infty}\frac{1}{m
^{d}}h^{0}(Y_{1}\cap\cdots\cap Y_{n-d}, m\epsilon A+m L\otimes
\mathcal{I}_{+}(m\varphi))>0 ,$$
by taking $s=0$ in $(4.2)$, we obtain
$$\lim\limits_{\epsilon\to 0}\varlimsup\limits_{m\rightarrow
\infty}\frac{1}{m^{n}\epsilon^{n-d}}h^{0}(X, m\epsilon A+m L\otimes
\mathcal{I}_{+}(m\varphi))>0 .$$
The proposition is proved.
\end{proof}

We now prove the formula $(4.4)$ promised in the above proposition.
\begin{lemme}
We have
$$\frac{1}{m^{n-s_{0}+1}}h^{0}(Y, m\epsilon A+mL\otimes
\mathcal{I}_{+}(m\varphi))=\frac{1}{m^{n-s_{0}+1}}e^{0,0}_{1}(\epsilon,m)
$$
$$=-\frac{1}{m^{n-s_{0}+1}}(\sum\limits_{q\geq
1}(-1)^{q}e_{1}^{0,q}(\epsilon,m))+O(\frac{1}{m}).$$
\end{lemme}

\begin{proof}
Thanks to $(iii),(iv)$ of Lemma 4.1 and the section 4 of \cite{K}, 
$\mathcal{O}_{Y}( mL\otimes \mathcal{I}_{+}(m\varphi))$ is resolved by
$$\mathcal{O}_{Y}(m\epsilon A+mL\otimes \mathcal{I}_{+}(m\varphi))\rightarrow
\oplus_{1\leq i\leq m\epsilon} \mathcal{O}_{Y\cap Y_{i}}(m\epsilon A+mL\otimes
\mathcal{I}_{+}(m\varphi))\leqno (*)$$
$$\rightarrow \oplus_{1\leq i_{1}<i_{2}\leq m\epsilon} \mathcal{O}_{Y\cap
Y_{i_{1}}\cap Y_{i_{2}}}(m\epsilon A+m L\otimes \mathcal{I}_{+}(m\varphi))$$
$$\rightarrow \cdots $$
Then
$$H^{k}(Y, mL\otimes \mathcal{I}_{+}(m\varphi))=\mathbb{H}^{k}(\epsilon,
m)\leqno (5.1)$$
where $\mathbb{H}^{k}(\epsilon,
m)$ represents the hypercohomology of $(*)$. 

We now calculate the asymptotic behaviour of on the both sides of $(5.1)$.
The Nadel vanishing theorem implies that
$$\lim\limits_{m\rightarrow \infty}\frac{1}{m^{n-s_{0}+1}}h^{k}(Y,
mL\otimes\mathcal{I}_{+}(m\varphi))=0\qquad\text{for any}\hspace{5 pt}k\geq 1.\leqno (5.2)$$ 
Moreover, since we assume that $ \nd (L, h) = d <  \dim Y$, 
we have
$$\lim\limits_{m\rightarrow \infty}\frac{1}{m^{n-s_{0}+1}} h^{0}(Y,
mL\otimes\mathcal{I}_{+}(m\varphi))=0 .\leqno (5.3)$$
By calculating the asymptotic cohomology on both sides of $(5.1)$,
the equations $(5.2)$ and $(5.3)$ imply in particular that
$$\lim\limits_{m\rightarrow
\infty}\frac{1}{m^{n-s_{0}+1}}\sum_{k}(-1)^{k}h^{k}(\epsilon, m)=0 ,\leqno
(5.4)$$
where $h^{k}(\epsilon, m)$ denotes the dimension of $\mathbb{H}^{k}(\epsilon,
m)$.
 
For the right side of $(5.1)$, we have
$$\lim\limits_{m\rightarrow
\infty}\frac{1}{m^{n-s_{0}+1}}\binom{m\epsilon}{q}h^{p}(Y\cap Y_{1}\cap\cdots
\cap
Y_{q}, m\epsilon A\otimes mL\otimes\mathcal{I}_{+}(m\varphi))=0$$
for every $p\geq 1$ by Nadel vanishing theorem.
If $ p = 0$, then
$$\binom{m\epsilon}{q}h^{0}(Y\cap Y_{1}\cap\cdots \cap
Y_{q}, m\epsilon A\otimes
mL\otimes\mathcal{I}_{+}(m\varphi))=e_{1}^{0,q}(\epsilon, m)$$
by definition.
Thus $(5.4)$ implies that 
$$\lim\limits_{m\rightarrow \infty}\frac{1}{m^{n-s_{0}+1}}(\sum\limits_{q\geq
0}(-1)^{q}e_{1}^{0,q}(\epsilon,m))=0\qquad\text{for any}\hspace{5 pt}\epsilon>0 .$$
Then
$$\frac{1}{m^{n-s_{0}+1}}h^{0}(Y, m\epsilon A+mL\otimes
\mathcal{I}_{+}(m\varphi))=\frac{1}{m^{n-s_{0}+1}}e^{0,0}_{1}(\epsilon,
m)$$
$$=-\frac{1}{m^{n-s_{0}+1}}(\sum\limits_{q\geq
1}(-1)^{q}e_{1}^{0,q}(\epsilon,m))+O(\frac{1}{m}).$$
\end{proof}

\begin{remark}
Let $L$ be a pseudo-effective line bundle on
a compact Kähler manifold. 
S.Boucksom defined in \cite{Bou} a concept of numerical dimension $\nd(L)$ of pseudo-effective
line bundles without metrics. 
Let $\varphi_{\min}$ be a positive metric
of $L$ with minimal singularity. 
Proposition 4.4 implies in particular that
$$\nd(L)\geq \nd(L,\varphi_{\min}).$$
The example 1.7 in \cite{DPS 94} tells us that we cannot hope the equality 
$$\nd(L)= \nd(L,\varphi_{\min}).$$
In that example, the line bundle $L$ is nef and $\nd(L)=1$.
But there exists a unique positive singular metric $h$ on $L$ with the curvature form
$$\frac{i}{2\pi}\Theta_{h} (L)=[C] ,$$ 
where $C$ is a curve on $X$. 
Therefore $\varphi_{\min}=h$. 
Moreover, since $h$ has analytic singularities, we have 
$\nd(L,\varphi_{\min})=0$ by definition.
Thus
$$\nd(L)>\nd(L,\varphi_{\min}).$$
\end{remark}

\begin{remark}
The example 3.6 in \cite{Tsu} tells us that we
cannot hope the following equality:
$$\sup_A\varlimsup\limits_{m\rightarrow \infty} \frac{\ln h^{0}(X, A+mL\otimes
\mathcal{I}(m\varphi))}{\ln m}=\nd (L,\varphi) ,$$
where $A$ runs over all the amples bundles on $X$.
In fact, H.Tsuji defined a closed positive $(1,1)$-current $T$ on
$\mathbb{P}^{1}$:
$$T=\sum_{i=1}^{+\infty}\sum_{j=1}^{3^{i-1}} \frac{1}{4^{i}} P_{i,j}$$
where $\{P_{i,j}\}$ are distinct points on $\mathbb{P}^{1}$.
There exists thus a singular metric $\varphi$ on $L=\mathcal{O}(1)$ with
$\frac{i}{2\pi}\Theta_{\varphi}(L)=T$.
It is easy to construct a quasi-equisingular approximation $\{\varphi_{k}\}$ to $\varphi$ such that 
$$\frac{i}{2\pi}\Theta_{\varphi_{k}}(L)=\sum_{i=1}^{k}\sum_{j=1}^{3^{i-1}} \frac{1}{4^{i}}
P_{i,j}+C^{\infty}.$$
Then $\nd (L,\varphi)=0$.

On the other hand, thanks to the construction of $\varphi$, we have
$$\varlimsup\limits_{m\rightarrow \infty} \frac{h^{0}(\mathbb{P}^{1},
\mathcal{O}(s+m)\otimes
\mathcal{I}(m\varphi))}{m}=\varlimsup\limits_{k\rightarrow \infty}
\frac{h^{0}(\mathbb{P}^{1}, \mathcal{O}(s+4^{k}-1)\otimes
\mathcal{I}((4^{k}-1)\varphi))}{4^{k}-1}$$
for every $s\geq 1$.
By construction, 
$$\mathcal{I}((4^{k}-1)\varphi)_{x}=\mathcal{O}_{x}$$
for $x\notin
\{P_{i,j}\}_{i\leq k-1}$, 
and $\mathcal{I}((4^{k}-1)\varphi)$ has multiplicity $\lfloor\frac{4^{k}-1}{4^{i}}\rfloor=4^{k-i}-1$ 
on $3^{i-1}$ points $\{P_{i,1},..., P_{i,3^{i-1}}\}$.
Therefore
$$h^{0}(\mathbb{P}^{1}, \mathcal{O}(s+4^{k}-1)\otimes
\mathcal{I}((4^{k}-1)\varphi))=s+4^{k}-\sum_{i=1}^{k-1}3^{i-1}(4^{k-i}-1)$$
$$=\frac{9}{2}3^{k-1}+s-\frac{1}{2}.$$
Then
$$\sup_A\varlimsup\limits_{m\rightarrow \infty} \frac{\ln h^{0}(\mathbb{P}^{1},
A+\mathcal{O}(m)\otimes \mathcal{I}(m\varphi))}{\ln m}=\frac{\ln 3}{\ln 4} .$$
Therefore
$$\nd(L,\varphi)\neq \sup_A\varlimsup\limits_{m\rightarrow \infty} \frac{\ln
h^{0}(\mathbb{P}^{1},
A+\mathcal{O}(m)\otimes \mathcal{I}(m\varphi))}{\ln m}.$$
\end{remark}

\section { A Kawamata-Viehweg-Nadel Vanishing Theorem}

% There is no good writing, only good re-writing.

In this section, 
we will prove that given a pseudo-effective
line bundle $(L,\varphi)$ on a compact Kähler manifold $X$ of dimension $n$,
then
$$H^{p}(X, K_{X}\otimes L\otimes \mathcal{I}_{+}(\varphi))=0\qquad\text{for}\hspace{5 pt}p\geq n-\nd (L,\varphi)+1 .$$
The main advantage of this version of the Kawamata-Viehweg-Nadel vanishing theorem is that 
we do not need any strict positivity of the line bundle. 
In our case, 
we just suppose that $X$ is a compact Kähler manifold and $\varphi$ is a
quasi-psh function such that $\frac{i}{2\pi}\Theta_{\varphi}(L)\geq 0$. 
But as a compensation, we need the multiplier ideal sheaf $\mathcal{I}_{+}(\varphi)$ instead of $\mathcal{I}(\varphi)$.

When $X$ is projective, the proof of our vanishing theorem is much easier. 
We first give a quick proof of this vanishing theorem in the projective case by the tools developed in the previous sections.
First of all, we prove the vanishing theorem in the case $\nd (L,\varphi)=\dim X$.
\begin{prop}
Let $ (L, \varphi) $ a pseudo-effective line bundle on a smooth projective
variety of dimension $ n $ with $ \nd (L, \varphi) = n $. 
Then
$$H^{i}(X, K_{X}+L\otimes\mathcal{I}_{+}(\varphi))=0\qquad\text{for any } i>0.$$
\end{prop}

\begin{proof}

Fix a smooth metric $h_{0}$ on $L$. $(L,\varphi)$ is pseudo-effective means that
$$\frac{i}{2\pi}\Theta_{\varphi}(L)=\frac{i}{2\pi}\Theta_{h_{0}}(L)+dd^{c}\varphi\geq 0 .$$
Since $ \frac{i}{2\pi} \Theta_{\varphi}(L) $ is not strictly positive, 
we need to add a portion of the metric $ \widetilde{\varphi} $ constructed in
Proposition 3.5
to make the new curvature become strictly positive.
We will see that this operation preserves $ \mathcal{I}_{+}(\varphi) $.

First of all, by the definition of $\mathcal{I}_{+}$, there is a $ \delta> 0 $
such that
$$\mathcal{I}_{+}(\varphi)=\mathcal{I}((1+\delta)\varphi) .$$
Let $ \widetilde{\varphi} $ be the psh function constructed in Propositon 3.5. 
For any $\epsilon> 0 $, we define a new metric
$$h_{0}e^{-(1+\sigma(\epsilon)-\epsilon)\varphi-\epsilon\widetilde{\varphi}}$$
on $L$
where $0 < \sigma(\epsilon)\ll \epsilon$.
Since $dd^{c}\varphi\geq -c\omega$ for some constant $c$,
\footnote{In our context, $\varphi$ is a function on $X$, then
$\frac{i}{2\pi}\Theta_{\varphi}(L)=\frac{i}{2\pi}\Theta_{h_{0}}(L)+dd^{c}\varphi\geq 0$. Therefore
$dd^{c}\varphi\geq -c\omega$.} the condition
$\sigma(\epsilon) \ll\epsilon$ implies that
$$\frac{i}{2\pi}\Theta_{(1+\sigma(\epsilon)-\epsilon)\varphi+\epsilon\widetilde{\varphi}}
(L)=(1+\sigma(\epsilon)-\epsilon)\frac{i}{2\pi}\Theta_{\varphi}(L)+\epsilon
\frac{i}{2\pi}\Theta_{\widetilde{\varphi}}(L)+\sigma(\epsilon)dd^{c}\varphi>0.$$
Applying the standard Nadel vanishing theorem on
$$(X,L, \mathcal{I}((1+\sigma(\epsilon)-\epsilon)\varphi+\epsilon\widetilde{\varphi}))$$
we get
$$H^{i}(X,
K_{X}+L\otimes\mathcal{I}
((1+\sigma(\epsilon)-\epsilon)\varphi+\epsilon\widetilde{\varphi}))=0\qquad\text{for}\hspace{5 pt} i>0.\leqno
(1.1)$$
On the other hand, it it not hard to prove that
$$\mathcal{I}_{+}(\varphi)=\mathcal{I}((1+\sigma(\epsilon)-\epsilon)\varphi+\epsilon\widetilde{\varphi})
\qquad\text{for }\epsilon\ll 1.\leqno (1.2)$$
We postpone the proof of $(1.2)$ in Lemma 5.2 and conclude first the proof of Proposition 5.1.
Taking $\epsilon$ small enough, $(1.1)$ and $(1.2)$ imply the proposition.
\end{proof}

\begin{lemme}
If $ \epsilon $ is small enough, then
$$\mathcal{I}_{+}(\varphi)=\mathcal{I}((1+\sigma(\epsilon)-\epsilon)\varphi+\epsilon\widetilde{\varphi}).$$
\end{lemme}

\begin{proof}
By Proposition 3.5, we have
$ \varphi_{m}\preccurlyeq \widetilde {\varphi} $. 
Therefore
$$(1+\sigma(\epsilon)-\epsilon)\varphi_{m}+\epsilon\varphi_{m}
\preccurlyeq (1+\sigma(\epsilon)-\epsilon)\varphi+\epsilon\widetilde{\varphi},$$
which implies that
$$\mathcal{I}((1+\sigma(\epsilon)-\epsilon)\varphi+\epsilon\widetilde{\varphi}
)\subset \mathcal{I}((1+\sigma(\epsilon))\varphi_{m}).\leqno(2.1)$$
Notice moreover that Lemma 3.1 implies that
$$\mathcal{I}((1+\sigma(\epsilon))\varphi_{m})\subset
\mathcal{I}_{+}(\varphi)\leqno(2.2)$$
for $ m $ large enough with respect to $ \sigma (\epsilon) $.
Combining $(2.1)$ and $(2.2)$, we obtain
$$\mathcal{I}((1+\sigma(\epsilon)-\epsilon)\varphi+\epsilon\widetilde{\varphi}
)\subset\mathcal{I}_{+}(\varphi).$$

For the other side inclusion, we take $ f \in \mathcal{I}_{+}(\varphi)_{x}$,
i.e., $f$ satisfies
$$\int_{U_{x}} \mid f\mid^{2}e^{-2(1+\delta)\varphi}< +\infty .\leqno (2.3)$$
We need to prove that
$$f\in \mathcal{I}((1+\sigma(\epsilon)-\epsilon)\varphi+\epsilon\widetilde{\varphi})_{x}.$$
Since $\widetilde{\varphi}$ is a quasi-psh function, by taking $ \epsilon$ small
enough, we have
$$\int_{U_{x}}e^{-2\frac{\epsilon}{\delta}\widetilde{\varphi}}< +\infty.\leqno
(2.4)$$
Therefore $(2.3)$ and $(2.4)$ imply that
$$\int_{U_{x}} \mid
f\mid^{2}e^{-2(1+\sigma(\epsilon)-\epsilon)\varphi-2\epsilon\widetilde{\varphi}}$$
$$\leq \int_{U_{x}} \mid f\mid^{2}e^{-2(1+\delta)\varphi}\int_{U_{x}}
e^{-2\frac{\epsilon}{\delta}\widetilde{\varphi}}
< +\infty$$
by H\"{o}lder's inequality.
Then $f\in \mathcal{I}((1+\sigma(\epsilon)-\epsilon)\varphi+\epsilon\widetilde{\varphi})$,
and we obtain
$$\mathcal{I}_{+}(\varphi)\subset
\mathcal{I}((1+\sigma(\epsilon)-\epsilon)\varphi+\epsilon\widetilde{\varphi})\qquad\text{for }\epsilon\ll 1.$$
The lemma is proved.

\end{proof}

Using Proposition 5.1, we now prove a Kawamata-Viehweg-Nadel type vanishing theorem by induction on dimension. 

\begin{prop}
Let $ (L, \varphi) $ be a pseudo-effective line bundle on a projective
variety $X$ of dimension $ n $. 
Then
$$H^{p}(X, K_{X}+L\otimes\mathcal{I}_{+}(\varphi))=0
\qquad\text{for}\hspace{5 pt} p\geq n-\nd(L,\varphi)+1.$$
\end{prop}

\begin{proof}

If $\nd(L, \varphi)=n$, the proposition has already been proved in Proposition 5.1.

Assume now that $\nd(L, \varphi) <n $. Let $ A $ be a sufficient ample line bundle large
enough with respect to $L$, 
and let $ S $ be the zero divisor of a general section of $H^{0}(X,A)$.
Let $ \epsilon> 0 $ be small enough such that the
condition of $(iv)$ of Lemma 4.1 is satisfied (by Lemma 4.1, such kind of
$\epsilon$ is independent of $A$ ! ) . 
Then we have an exact sequence:
$$0\rightarrow \mathcal{I}_{+}(\varphi)\otimes \mathcal{O}(-S)\rightarrow
\adj_{S}^{\epsilon}(\varphi)\rightarrow \mathcal{I}_{+}(S,
\varphi_{S})\rightarrow 0 .\leqno (3.1)$$
By $(iv)$ of Lemma 4.1, we have 
$$\adj_{S}^{\epsilon}(\varphi)=\mathcal{I}_{+}(\varphi) .$$
Then $(3.1)$ induces the following exact sequence:
$$H^{q}(S, K_{S}+L\otimes \mathcal{I}_{+}(\varphi|_{S}))\rightarrow
H^{q+1}(X,K_{X}+L\otimes \mathcal{I}_{+}(\varphi))\rightarrow H^{q+1}(X,K_{X}+ A
+ L\otimes \mathcal{I}_{+}(\varphi)),$$
for every $q$.
Taking $A$ to be sufficient ample, we have 
$$ H^{q+1}(X,K_{X}+ A + L\otimes
\mathcal{I}_{+}(\varphi))=0$$
by Nadel vanishing theorem. 
Thus the above exact sequence implies that
$$H^{q}(S, K_{S}+L\otimes \mathcal{I}_{+}(\varphi|_{S}))\rightarrow
H^{q+1}(X,K_{X}+L\otimes \mathcal{I}_{+}(\varphi))$$
is surjective.
The proposition is proved by induction on dimension.

\end{proof}

\begin{remark}
We can also prove that
$$H^{i}(X, K_{X}+A+L\otimes \adj_{S}^{\epsilon}(\varphi))=0 
\qquad\text{for}\hspace{5 pt} i>0 ,$$
which gives another proof of this theorem.
\end{remark}

The main goal of this section is to prove our vanishing theorem in the Kähler case. 
For this, we combine the methods developed in \cite{Mou} and \cite{DP}. 
First of all, we prove that $\mathcal{I}_{+}$ is essentially with analytic singularities. More precisely,

\begin{lemme}
Let $(L, \varphi)$ be a pseudo-effective line bundle on a compact kähler
manifold $X$. 
Then there exists a quasi-equisingular approximation $\{\varphi_{k}\}$ of
$\varphi$ such that 
$$\mathcal{I}((1+\frac{2}{k})\varphi_{k})=\mathcal{I}_{+}(\varphi)
\qquad\text{for}\hspace{5 pt} k\gg 1.\leqno (4.1)$$
\end{lemme}

\begin{proof}
By \cite{DPS 01}, there exists a quasi-equisingular approximation $\{\varphi_{k}\}$
of $\varphi$. The comparison of integrals techniques discussed in \cite{DPS 01}
implies that we can choose a subsequence $\{\varphi_{f(k)}\}$ such that
$$\mathcal{I}((1+\frac{2}{k})\varphi_{f(k)})\subset\mathcal{I}_{+}
(\varphi).\leqno (4.1)$$
In fact, if $X$ is projective, we take $s=1+\epsilon$ and a subsequence
$\{f(k)\}_{k=1}^{\infty}$ for $f(k)\gg k$ in Lemma 3.1. 
Using Lemma 3.1, we get the inclusion $(4.1)$. 
If $X$ is just a compact Kähler manifold, we can get the same inclusion for some
Stein neighborhood. 
Using the glueing techniques, we obtain also the inclusion $(4.1)$. (see \cite{DPS
01} for details)

On the other hand, since $\varphi_{k}$ is less singular than $\varphi$,
the definition of $\mathcal{I}_{+}(\varphi)$ implies that
$$\mathcal{I}((1+\frac{2}{k})\varphi_{f(k)})\supset\mathcal{I}_{+}(\varphi)
\qquad\text{for}\hspace{5 pt} k\gg 1.$$
The lemma is thus proved.

\end{proof}

The following lemma will be important in the proof of our Kawamata-Viehweg-Nadel
vanishing theorem. 
The advantage of the lemma is that to prove the convergence in higher degree cohomology with multiplier ideal sheaves, 
we just need to check the convergence for some smooth metric.

We first fix some notations. Let $(L,\varphi)$ be a pseudo-effective line
bundle on a compact Kähler manifold $X$ and
let $\mathcal{U}=\{U_{\alpha}\}_{\alpha\in I}$ be a Stein covering of $X$. 
We denote $U_{\alpha_{0}\alpha_{1}\cdots\alpha_{q}}=U_{\alpha_{0}}\cap\cdots\cap U_{\alpha_{q}}$ and
$ C^{q}(\mathcal{U}, K_{X}\otimes L\otimes \mathcal{I}_{+}(\varphi)) $ 
the $\check{C}$ech $q$-cochain associated to $K_{X}\otimes L\otimes \mathcal{I}_{+}(\varphi)$.
For an element $c\in C^{q}(\mathcal{U}, K_{X}\otimes L\otimes \mathcal{I}_{+}(\varphi))$, 
we denote its component on $U_{\alpha_{0}\alpha_{1}\cdots\alpha_{q}}$ by $c_{\alpha_{0}\alpha_{1}\cdots\alpha_{q}}$.
Let
$$\delta_{p}: C^{p-1}(\mathcal{U}, \mathcal{I}_{+}(\varphi))\rightarrow C^{p}(\mathcal{U}, \mathcal{I}_{+}(\varphi))$$
be the  $\check{C}$ech operator,
and $ Z^{p}(\mathcal{U}, \mathcal{I}_{+}(\varphi))=\Ker \delta_{p+1}$.

\begin{lemme}
Let $L$ be a line bundle on a compact Kähler manifold $X$ 
and $\varphi$ a singular metric on $L$. 
Let $\{U_{\alpha}\}_{\alpha\in I}$ be a Stein covering of $X$.
Let $u$ be an element in $\check{H}^{p}(X, K_{X}+L\otimes\mathcal{I}_{+}(\varphi))$.
If there exists a sequence 
$\{v_{k}\}_{k=1}^{\infty} \subset C^{p}(\mathcal{U}, K_{X}\otimes L\otimes \mathcal{I}_{+}(\varphi))$ 
in the same cohomology class as $u$ satisfying the key $L^{2}$ convergent condition:
$$\lim\limits_{k\to \infty}\int_{U_{\alpha_{0}...\alpha_{p}}} |v_{k, \alpha_{0}...\alpha_{p}}|
^{2} \rightarrow 0 ,\leqno (5.1)$$
where the $L^{2}$ norm $|v|^{2}$ in $(5.1)$ is taken for some fixed smooth metric on
$L$, 
then $u=0$ in $\check{H}^{p}(X, K_{X}+L\otimes\mathcal{I}_{+}(\varphi))$.
\end{lemme}

\begin{proof}

On the $p$-cochain space $C^{p}(\mathcal{U}, \mathcal{I}_{+}(\varphi))$, we first
define a familly of natural semi-norms:
for $f\in C^{p}(\mathcal{U}, \mathcal{I}_{+}(\varphi))$,
we define a family of semi-norms: 
$$\{ \int_{V} |f|^{2}\omega^{n} | \text{ any open set }V\Subset U_{\alpha_{0}...\alpha_{p}}\} .$$
We claim that 
$C^{p}(\mathcal{U}, \mathcal{I}_{+}(\varphi))$ is a Fréchet space with respect
to the family of semi-norms as above. 

Proof of the claim:
We need to prove that if $f_{i}\in \mathcal{I}_{+}(\varphi)$ and $
f_{i}\rightarrow f_{0}$ in the family of the above semi-norms, then
$ f_{0}\in \mathcal{I}_{+}(\varphi)$. 
First of all, by the definition of the semi-norms,  $f_{0}$ is holomorphic. 
By Lemma 5.4 we can choose a quasi-psh function $\psi$ with analytic singularities such that
$$\mathcal{I}(\psi)=\mathcal{I}_{+}(\varphi).$$
Let $\pi : X_{k}\rightarrow X$ be a log resolution of $\psi$.
Then the current $E=\lfloor dd^{c}(\psi\circ\pi)\rfloor$ has normal crossing
singularities.
Since $f_{i}\in \mathcal{I}_{+}(\varphi)=\mathcal{I}(\psi)$, 
we have
$$(f_{i}\circ\pi )\cdot J\in \mathcal{O}(-E),\leqno (5.2)$$
where $J$ is the Jacobian of $\pi$. 
Since $f_{i}\circ\pi\rightharpoonup f_{0}\circ\pi$ in the sense of weak
convergence and $E$ has normal crossing singularities, 
$(5.2)$ implies that
$$(f_{0}\circ\pi )\cdot J\in \mathcal{O}(-E) .$$ 
Therefore $f_{0}\in\mathcal{I}_{+}(\varphi)$. 
The claim is proved.
\vspace{5 pt}

As a consequence, the $\check{C}$ech operator $\delta_{p}$ is continuous 
and its kernel $Z^{p-1}(\mathcal{U}, \mathcal{I}_{+}(\varphi))$ 
is also a Fréchet space.
Therefore we have a refined continuous morphism between the following two Fréchet
spaces:
$$\delta_{p}: C^{p-1}(\mathcal{U}, \mathcal{I}_{+}(\varphi))\rightarrow
Z^{p}(\mathcal{U}, \mathcal{I}_{+}(\varphi)).$$
Since the cokernel of $\delta_{p}$ is $\check{H}^{p}(X, (K_{X}+L)\otimes 
\mathcal{I}_{+}(\varphi))$ which is of finite dimension, 
by the standard Fredholm theory, the image of $\delta_{p}$ is closed. 
Thus the quotient morphism
$$ Z^{p}(\mathcal{U}, \mathcal{I}_{+}(\varphi))\rightarrow
\frac{Z^{p}(\mathcal{U}, \mathcal{I}_{+}(\varphi))}{\image (\delta_{p})}\leqno
(5.3)$$
is continuous.
By observing $\frac{Z^{p}(\mathcal{U}, \mathcal{I}_{+}(\varphi))}{\image
(\delta_{p})}=\check{H}^{p}(X,K_{X}+L\otimes  \mathcal{I}_{+}(\varphi))$,
$(5.3)$ implies that
$$Z^{p}(\mathcal{U}, \mathcal{I}_{+}(\varphi))\rightarrow
\check{H}^{p}(X,K_{X}+L\otimes  \mathcal{I}_{+}(\varphi))\leqno (5.4)$$
is continuous.

Thanks to the claim, the condition $(5.1)$ implies that $\{v_{k}\}_{k=1}^{\infty}$ 
tends to
$0$ in the Fréchet space $Z^{p}(\mathcal{U}, \mathcal{I}_{+}(\varphi))$. 
Therefore their images in $\check{H}^{p}(X, K_{X}\otimes L\otimes
\mathcal{I}_{+}(\varphi))$ tend to $0$ by the continuity of $(5.4)$.
Since by construction their images are in the same class $[u]$, 
we conclude 
$u=0$ in $\check{H}^{p}(X, K_{X}\otimes L\otimes \mathcal{I}_{+}(\varphi))$.

\end{proof}

\begin{remark}
The essential point here is that for the quasi-psh function $\varphi$ without analytic singularities, 
we should consider $\mathcal{I}_{+}(\varphi)$ instead of $\mathcal{I}(\varphi)$.
We do not know whether this lemma is also true for $\mathcal{I}(\varphi)$. 
\end{remark}

We prove Theorem 1.3 in the rest of this section. 
Since $\varphi$ has not necessarily analytic singularities, 
this makes troubles when we use $L^{2}$ estimates.
Therefore we replace $\varphi$ by a quasi-equisingular approximation. 
Thanks to Lemma 5.4,  
we can keep $\mathcal{I}_{+}(\varphi)$ 
by the metrics $(1+\frac{2}{k})\varphi_{k}$ with analytic singularities.
We should also use a Monge-Ampère equation to construct other metrics $ \widehat{\varphi}_{k}$ 
of which we can control the eigenvalues.
Therefore we can use $L^{2}$ estimates for every $ \widehat{\varphi}_{k}$. By some
delicate analysis, we can prove the theorem.
This idea comes from \cite{DP} and \cite{Mou}.
We will construct the key metrics $ \widehat{\varphi}_{k}$ in Lemma 5.6
and prove some important properties of $ \widehat{\varphi}_{k}$ in Lemma 5.7 and
5.8.
We prove finally the vanishing theorem in Proposition 5.9.

\begin{lemme}
Let $(L, \varphi)$ be a pseudo-effective line bundle on a compact
kähler manifold $(X, \omega )$ and
let $p\geq n-\nd (L,\varphi) +1$.
Then there exists a sequence of new metrics $\{ \widehat{\varphi}_{k}
\}_{k=1}^{\infty}$ with analytic singularities on $L$ satisfying the following
properties:

$(i)$ $\mathcal{I}(\widehat{\varphi}_{k})=\mathcal{I}_{+}(\varphi)$ for all $k$.

$(ii)$ Let $\lambda_{1, k}\leq\lambda_{2, k}\leq\cdots\leq\lambda_{n, k}$ be the
eigenvalues of $\frac{i}{2\pi}\Theta_{\widehat{\varphi}_{k}}(L)$ 
with respect to the base metric $\omega$. 
Then there exist two sequences
$\tau_{k}\rightarrow 0, \epsilon_{k}\rightarrow 0$ such that  
$$\epsilon_{k}\gg \tau_{k}+\frac{1}{k} \qquad\text{and}
 \qquad \lambda_{1, k}(x)\geq
-\epsilon_{k}-\frac{C}{k}-\tau_{k}$$
for all $x\in X$ and $k$, where $C$ is a constant independent of $k$.

$(iii)$ We can choose $\beta>0$ and $0< \alpha <1$ independent of $k$ such that
for every $k$, there exists an open subset $U_{k}$ of $X$ satisfying 
$$\vol
(U_{k})\leq \epsilon_{k}^{\beta}\qquad\text{and}\qquad
\lambda_{p}+2\epsilon_{k}\geq \epsilon_{k}^{\alpha} \hspace{5 pt}\text{ on } X\setminus U_{k}.$$
\end{lemme}

\begin{proof}

Recall that we first fix a smooth metric $h_{0}$ on $L$.
Then given a new metric $\varphi$ as a form of function, 
we just means that the hermitian metric form on $L$ is $h_{0}e^{-\varphi}$.

By definition, there exists a $s_{1} > 0$ such that
$$\mathcal{I}_{+}(\varphi)=\mathcal{I}((1+s_{1})\varphi). \leqno (*)$$
Let $\varphi_{k}$ be the quasi-equisingular approximation of $\varphi$ in Lemma 5.4.
Then there is a positive sequence $\tau_{k}\rightarrow 0$ such that 
$$\frac{i}{2\pi}\Theta_{\varphi_{k}}(L)\geq -\tau_{k} \omega \qquad
\text{and}
 \qquad
\mathcal{I}((1+\frac{2}{k})\varphi_{k})=\mathcal{I}_{+}(\varphi)\leqno (6.1)$$
for every $k$. 
We can also choose a positive sequence $\epsilon_{k}\rightarrow 0$ such that
$\epsilon_{k}\gg \tau_{k}+\frac{1}{k}$.

We begin to construct new metrics by solving a Monge-Ampère equation.
Let $\pi: X_{k}\rightarrow X$ be a log resolution of $\varphi_{k}$. 
Then $dd^{c}(\varphi_{k}\circ \pi)$ is of the form $[E_{k}]+C^{\infty}$ 
where $[E_{k}]$ is a normal crossing $\mathbb{Q}$-divisor.
Let $Z_{k}=\pi_{*}( E_{k} )$.
By \cite{Bou}, there exists a smooth metric $h_{k}$ on $[E_{k}]$, such that for all
$\delta>0$ small enough, 
$$\pi^{*}(\omega)+ \delta\frac{i}{2\pi}\Theta_{h_{k}}(-E_{k})$$
is a Kähler form on $X_{k}$.
Therefore, fixing a positive sequence $\delta_{k}\rightarrow 0$,
we can solve a Monge-Ampère equation on $X_{k}$:
$$((\frac{i}{2\pi}\pi^{*}\Theta_{\varphi_{k}}(L))_{\ac}+\epsilon_{k}\pi^{*}\omega+\delta_{k}\frac{i}{2\pi}
\Theta_{h_{k}}(-E_{k})+dd^{c}\psi_{k,\epsilon, \delta_{k}})^{n}\leqno (6.2)$$
$$=C(k,\delta,\epsilon)\cdot \epsilon_{k}^{n-d}(\omega+\delta_{k}\frac{i}{2\pi}\Theta_{h_{k}}(-E_{k}))^{n}$$
with the condition
$$\sup_{z\in X_{k}} (\varphi_{k}\circ\pi+\psi_{k,\epsilon, \delta_{k}}+\delta_{k}\ln |E_{k}|_{h_{k}})(z)=0$$
where $d=\nd(L,\varphi)$.
Thanks to the definition of the numerical dimension, there exists a
uniform constant $C>0$ such that $C(k,\delta,\epsilon)\geq C$.
By observing moreover that
$$i\partial\overline{\partial}\ln
|E_{k}|_{h_{k}}=[E_{k}]+\frac{i}{2\pi}\Theta_{h_{k}}(-E_{k}),$$
the equation $(6.2)$ implies
$$\frac{i}{2\pi}\Theta_{\varphi_{k}+\psi_{k,\epsilon, \delta_{k}}+\delta_{k}\ln
|E_{k}|_{h_{k}}}(\pi^{*}L)\geq -\epsilon_{k}\omega .\leqno (6.3)$$

We define now a new metric on $(X_{k}, \pi^{*}L)$ by 
(i.e. $h_{0}e^{-\widehat{\varphi}_{k}}$ as a metric form ! )
$$\widehat{\varphi}_{k}=(1+\frac{2}{k}-s)\varphi_{k}\circ\pi+s(\varphi_{k}
\circ\pi+\psi_{k,\epsilon, \delta}+\delta\ln |E_{k}|_{h_{k}})\leqno (6.4)$$
where $ s $ is a constant sufficient small with respect to $ s_{1}$. 
$s$ will be made 
precise in Lemma 5.7.
Then
$$\frac{i}{2\pi}\Theta_{\widehat{\varphi}_{k}}(\pi^{*}L)=(1-s)\frac{i}{2\pi}\Theta_{\varphi_{k}}(\pi^{*}L)+s
\frac{i}{2\pi}\Theta_{\varphi_{k}+\psi_{k,\epsilon, \delta_{k}}+\delta_{k}\ln
|E_{k}|_{h_{k}}}(\pi^{*}L)+\frac{2}{k}dd^{c}\varphi_{k}. \leqno (6.5)$$
$(6.3)$ gives the estimate for the second term of the right hand side of $(6.5)$.
For the last term of the right hand side of $(6.5)$, since
$\varphi_{k}$ is a function on $X$ satisfying
$$\frac{i}{2\pi}\Theta_{\varphi_{k}}(L)=\frac{i}{2\pi}\Theta_{h_{0}}(L)+dd^{c}\varphi_{k}\geq -c\omega ,$$
we have
$$dd^{c}\varphi_{k}\geq -C\omega$$ 
for some uniform constant $C$.
Then
$$\frac{i}{2\pi}\Theta_{\widehat{\varphi}_{k}}(\pi^{*}L)\geq -\epsilon_{k}\omega-\tau_{k}\omega-\frac{C}{k}\omega .\leqno (6.6)$$
Thus $\widehat{\varphi}_{k}$ induces a quasi-psh function on $X$ 
by extending it from $X\setminus Z_{k}$ to the
whole $X$. 
It is the metric that we want to construct. 
We denote it also $\widehat{\varphi}_{k}$ for simplicity.
We will prove properties $(i)$ to $(iii)$ in Lemma 5.7 and Lemma 5.8.
\end{proof}

\begin{lemme}
If we take $ s $ in $(6.4)$ small enough with respect to $s_{1}$ in $(*)$ of Lemma 5.6,  
then
$$\int_{U} |f|^{2} e^{-2\widehat{\varphi}_{k}}\leq C_{|f|_{L^{\infty}}}\int_{U}
(|f|^{2} e^{-2(1+s_{1})\varphi})^{\frac{1}{1+s_{1}}}$$
for all $U$ in $X$ and $k\gg 1$, where
$C_{|f|_{L^{\infty}}}$ is a constant depending only on $|f|_{L^{\infty}}$
(in particular, it is independent of the open subset $U$ and $k$).
As a consequence, we have
$$\mathcal{I}(\widehat{\varphi}_{k})=\mathcal{I}_{+}(\varphi)
\qquad\text{for any}\hspace{5 pt} k.$$
\end{lemme}

\begin{proof}

Thanks to $(6.3)$ in Lemma 5.6, we can extend
$\varphi_{k}+\psi_{k,\epsilon, \delta_{k}}+\delta\ln |E_{k}|_{h_{k}}$
to the whole $X$, 
satisfying the same inequality. 
Then the condition 
$$\sup_{x\in X} (\varphi_{k}+\psi_{k,\epsilon, \delta}+\delta_{k}\ln |E_{k}|_{h_{k}})(x)=0$$
and $(6.3)$ in Lemma 5.6 imply the existence of $a>0$ such that
$$\int_{X} e^{-2a(\varphi_{k}+\psi_{k,\epsilon, \delta}+\delta_{k}\ln |E_{k}|_{h_{k}})}$$ 
is uniformly bounded for all $k$.

By H\"{o}lder's inequality and the construction $(6.4)$ in Lemma 5.6, we have
$$\int_{U} |f|^{2} e^{-2\widehat{\varphi}_{k}}\leq \leqno (7.1)$$
$$(\int_{U} |f|^{2} e^{-2(1+s_{1})\varphi_{k}})^{\frac{1}{1+s_{1}}}(\int_{U}
|f|^{2} e^{-\frac{2s(1+s_{1})}{s_{1}}(\varphi_{k}+\psi_{k,\epsilon,
\delta_{k}}+\delta_{k}\ln |E_{k}|_{h_{k}})})^{\frac{s_{1}}{1+s_{1}}}$$
for $k\gg 1$, where $U$ is any open subset of $X$.
If we take $ s>0 $ satisfying $\frac{s(1+s_{1})}{s_{1}}\leq a$,
then the uniform boundness of 
$\int_{X} e^{-2a(\varphi_{k}+\psi_{k,\epsilon, \delta_{k}}+\delta_{k}\ln |E_{k}|_{h_{k}})}$ implies that
$$\int_{U} |f|^{2} e^{-\frac{2s(1+s_{1})}{s_{1}}(\varphi_{k}+\psi_{k,\epsilon,
\delta_{k}}+\delta_{k}\ln |E_{k}|_{h_{k}})}\leq C \cdot |f|_{L^{\infty}}\leqno
(7.2)$$
for any $U\subset X$ and $k\gg 1$.
Combining $(7.1)$ and $(7.2)$, we have
$$\int_{U} |f|^{2} e^{-2\widehat{\varphi}_{k}}\leq C_{|f|_{L^{\infty}}}(\int_{U}
|f|^{2} e^{-2(1+s_{1})\varphi_{k}})^{\frac{1}{1+s_{1}}}\leqno (7.3)$$
$$\leq C_{|f|_{L^{\infty}}}(\int_{U} |f|^{2} e^{-2(1+s_{1})\varphi})^{\frac{1}{1+s_{1}}}.$$
for some constant $C_{|f|_{L^{\infty}}}$ independent of the open subset
$U$ and $k\gg 1$.

We now prove the equality 
$\mathcal{I}(\widehat{\varphi}_{k})=\mathcal{I}_{+}(\varphi)$.
The inclusion 
$\mathcal{I}(\widehat{\varphi}_{k})\supset \mathcal{I}_{+}(\varphi)$
comes directly from $(7.3)$. 
By the construction, 
$\widehat{\varphi}_{k}$ is more singular than $(1+\frac{2}{k})\varphi_{k}$.
Then equality $(6.1)$ in Lemma 5.6 
implies that
$\mathcal{I}(\widehat{\varphi}_{k})\subset \mathcal{I}_{+}(\varphi)$.
The lemma is thus proved.

\end{proof}

The following lemma was essentially proved in \cite{Mou}.
\begin{lemme}
The new metrics $\{ \widehat{\varphi}_{k} \}_{k=1}^{\infty}$ satisfy
properties $(ii)$ and $(iii)$ in Lemma 5.6.
\end{lemme}

\begin{proof}
Let $\lambda_{1}(z)\leq\lambda_{2}(z)\leq\cdots\leq\lambda_{n}(z)$ be the eigenvalues of
the new metric $(X, L, \widehat{\varphi}_{k})$ 
with respect to the base metric $\omega$ for simplicity, 
( i.e. $\lambda_{i}$ here is equal to $\lambda_{i, k}$ in Lemma 5.6, 
since the proof is for fixed $k$, the simplification here will not lead misunderstanding.)
By $(6.6)$ in Lemma 5.6, we have
$$\lambda_{i}(z)\geq -\epsilon_{k}-\frac{C}{k}-\tau_{k} .$$
Let $\widehat{\lambda}_{i}=\lambda_{i}+2\epsilon_{k}$. 
Since $s$ is a fixed positive constant, the Monge-Ampère equation $(6.2)$ implies that
$$\prod_{i=1}^{n}\widehat{\lambda}_{i}(z)\geq C(s)\epsilon_{k}^{n-d}\leqno (8.1)$$
where $C(s)>0$ does not depend on $k$. 
Since $p> n-d$, we can take $\alpha$ such that $0 < \alpha <1$ and $n-d < \alpha p$. 

Let 
$$U_{k}=\{ z \in X \mid  \widehat{\lambda}_{p}(z)< \epsilon_{k}^{\alpha}\} .$$
Since $\epsilon_{k}\gg \tau_{k}+\frac{1}{k}$, we have
$\widehat{\lambda}_{i}(z)=\lambda_{i}(z)+2\epsilon_{k}\geq 0$ for any $z$ and $i$.
Thus the cohomological condition
$$\int_{X}
(\widehat{\lambda}_{1}+\widehat{\lambda}_{2}+\cdots+\widehat{\lambda}_{n}
)\omega^ { n } \leq M$$
implies that 
$$\int_{U_{k}}(\widehat{\lambda}_{1}+\widehat{\lambda}_{2}+\cdots+\widehat{
\lambda } _ { n})\omega^{n}\leq M .\leqno (8.2)$$
Observe that $(8.1)$ and the definition of $U_{k}$
imply that
$$\prod\limits_{p+1\leq i\leq n}\widehat{\lambda}_{i}(z)\geq
C(s)\frac{\epsilon_{k}^{n-d}}{\epsilon_{k}^{\alpha p}}
\qquad\text{for}\hspace{5 pt} z\in U_{k} .$$
Then
$$\sum\limits_{p+1\leq i\leq n}\widehat{\lambda}_{i}(z)\geq
C(\frac{\epsilon_{k}^{n-d}}{\epsilon_{k}^{\alpha p}})^{\frac{1}{n-p}}
\qquad\text{for}\hspace{5 pt} z\in U_{k} \leqno (8.3)$$
by the inequality of arithmetic and geometric means.
Applying $(8.3)$ to $(8.2)$, we have
$$\int_{U_{k}} (\frac{\epsilon_{k}^{n-d}}{\epsilon_{k}^{\alpha
p}})^{\frac{1}{n-p}}\omega^{n}\leq M' .\leqno (8.4)$$
Since $n-d < \alpha p $, $(8.4)$ implies the existence of $\beta > 0$ such that 
$$\vol (U_{k})\leq \epsilon_{k}^{\beta}.$$
The lemma is proved.
\end{proof}

We now prove the final conclusion.
\begin{prop}
Let $(L, \varphi)$ be a pseudo-effective line bundle on a compact kähler manifold $(X, \omega )$. 
Then
$$ H^{p}(X, K_{X}\otimes L\otimes \mathcal{I}_{+}(\varphi))=0
\qquad\text{for}\hspace{5 pt} p\geq n-\nd (L,\varphi)+1.$$
\end{prop}

\begin{remark}
One of the reason to use
$\mathcal{I}_{+}(\varphi)$ instead of $\mathcal{I}(\varphi)$ is that it does not seem
to be easy to prove that
$$H^{p}(X, K_{X}\otimes L\otimes \mathcal{I}(\varphi))=0\qquad\text{for}\hspace{5 pt}p\geq 1$$
even when $X$ is projective and $\nd(L,\varphi)=\dim X$.
\end{remark}

\begin{proof}

We prove it in two steps.

\textbf{Steps 1}: $L^{2}$ Estimates

Let $\{ \widehat{\varphi}_{k} \}_{k=1}^{\infty}$ be the metrics constructed in
Lemma 5.6,
and let $[u]$ be any element in $H^{p}(X, K_{X}\otimes L\otimes
\mathcal{I}_{+}(\varphi))$.
Let $f$ be a smooth $(n,p)$-form representing $[u]$. 
Then 
$$\int_{X} |f|^{2} e^{-2(1+s_{1})\varphi}< +\infty ,$$
for the constant $s_{1}$ in $(*)$ of Lemma 5.6.
Lemma 5.7 implies that 
$$\int_{U} |f|^{2} e^{-2\widehat{\varphi}_{k}}\leq C (\int_{U} |f|^{2}
e^{-2(1+s_{1})\varphi})^{\frac{1}{1+s_{1}}}\leqno (9.1)$$
for any open subset $U$ of $X$ and $k\gg 1$, where $C$ is a constant  
independent of $U$ and $k$ (but certainly depends on
$|f|_{L^{\infty}}$).
We now use the $L^{2}$ method in \cite{DP} to get a key estimate: 
we can write 
$$f=\overline{\partial} u_{k}+v_{k}\leqno (9.2)$$ 
with the following estimate
$$\int_{X}
|u_{k}|^{2}e^{-2\widehat{\varphi}_{k}}+\frac{1}{2p\epsilon_{k}}\int_{X}
|v_{k}|^{2}e^{-2\widehat{\varphi}_{k}}\leq \int_{X}
\frac{1}{\widehat{\lambda}_{1, k}+\widehat{\lambda}_{2, k}+\cdots+\widehat{\lambda}_{p, k
}}|f|^{2}e^{-2\widehat{\varphi}_{k}},\leqno (9.3)$$
for $\widehat{\lambda}_{i, k}=\lambda_{i, k}+2\epsilon_{k}$.
This comes from the Bochner inequality:
$$\|\overline{\partial}
u\|^{2}_{\widehat{\varphi}_{k}}+\|\overline{\partial}^{*}
u\|^{2}_{\widehat{\varphi}_{k}}\geq \int_{X-Z_{k}}
(\widehat{\lambda}_{1, k}+\widehat{\lambda}_{2, k}+\cdots+\widehat{\lambda}_{p, k}
-C\epsilon_{k})|u|_{\widehat{\varphi}_{k}} ^{2}dV$$
where $Z_{k}$ is the singular locus of $\varphi_{k}$ in
$X$.
(see \cite{DP} or the appendix for details)

Using $(9.3)$, we claim that
$$\lim\limits_{k\to \infty}\int_{X} |v_{k}|^{2} e^{-2\widehat{\varphi}_{k}}\rightarrow 0.\leqno (9.4)$$
Proof of the claim: 
Properties $(ii)$, $(iii)$ in Lemma 5.6 and $(9.3)$ imply that
$$\int_{X}
|u_{k}|^{2}e^{-2\widehat{\varphi}_{k}}+\frac{1}{2p\epsilon_{k}}\int_{X}
|v_{k}|^{2}e^{-2\widehat{\varphi}_{k}}$$
$$\leq
\int_{X}\frac{C_{1}}{\epsilon_{k}^{\alpha}}|f|^{2}e^{-2\widehat{\varphi}_{k}}
+\int_{U_{k}}\frac{1}{C_{2}\epsilon_{k}}|f|^{2}e^{-2\widehat{\varphi}_{k}}.$$
Then 
$$\int_{X} |v_{k}|^{2}e^{-2\widehat{\varphi}_{k}}\leq
C_{3}\epsilon_{k}^{1-\alpha}\int_{X} |f|^{2}e^{-2\widehat{\varphi}_{k}}+ C_{4}
\int_{U_{k}} |f|^{2}e^{-2\widehat{\varphi}_{k}}.\leqno (9.5)$$
Since $\vol (U_{k})\rightarrow 0$ by property $(iii)$ of Lemma 5.6,
$(9.1)$ implies that the second term of the right hand side of $(9.5)$ tends to $0$.
Reminding that $0<\alpha< 1$ and $\epsilon_{k}\rightarrow 0$ as $k\rightarrow \infty$, 
$(9.1)$ implies thus that the first term of the right hand side of $(9.5)$ also tends to $0$.
The claim is proved.

\textbf{Step 2}: Final step

We use Lemma 5.5 to obtain the final conclusion. 
Let $\mathcal{U}=\{U_{\alpha}\}_{\alpha\in I}$ be a Stein covering of $X$.
Thanks to $(9.4)$,  
we get a $p$-cycle representative of each $v_{k}$ by solving $\overline{\partial}$-equations,
i.e., $v_{k}$ can be writen as
$$v_{k}=\{ v_{k,\alpha_{0}...\alpha_{p}}\}\in C^{p}(\mathcal{U}, K_{X}\otimes
L\otimes \mathcal{I}(\widehat{\varphi}_{k}))$$
satisfying the $L^{2}$ conditions:
$$\int_{U_{\alpha_{0}...\alpha_{p}}}
|v_{k,\alpha_{0}...\alpha_{p}}|^{2}e^{-2\widehat{\varphi}_{k}}\leq C \int_{X}
|v_{k}|^{2}e^{-2\widehat{\varphi}_{k}}\leqno (9.6)$$
where $C$ does not depend on $k$.
$(9.6)$ and property $(i)$ in Lemma 5.6 imply that $\{v_{k}\}_{k=1}^{\infty}$ 
are all in $C^{p}(\mathcal{U}, K_{X}\otimes L\otimes\mathcal{I}_{+}(\varphi))$.

Since $\widehat{\varphi}_{k}\leq 0$, $(9.4)$ and $(9.6)$ imply that
$$\lim\limits_{k\to \infty}\int_{U_{i_{0}...i_{p}}} |v_{k,i_{0}...i_{p}}|^{2}=0.\leqno (9.7)$$
By $(9.2)$, $\{v_{k}\}_{k=1}^{\infty}$ are in the same cohomology of $u$.
Using Lemma 5.5, $(9.7)$ implies that $[u]=0$.
Since we choose $[u]$ as any element in $H^{p}(X, K_{X}\otimes L\otimes \mathcal{I}_{+}(\varphi))$, 
the proposition is proved.

\end{proof}

\section{Appendix}
For the convenience of reader, 
we give the proof of estimate $(9.3)$ in Proposition 5.9, but the proof is just extracted from \cite{DP}.
\begin{prop}
Let $(X,\omega)$ be a compact Kähler manifold and 
let $(L,h_{0}e^{-\varphi})$ be a line bundle on $X$ where $h_{0}$ is a smooth metric on $L$
and the quasi-psh function $\varphi$ has analytic singularities and smooth outside a subvariety $Z$. 
Assume that 
$$\frac{i}{2\pi}\Theta_{\varphi}(L)\geq -\epsilon\omega$$
on $X\setminus Z$, 
and $f$ is a smooth $L$-valued $(n,p)$-form satisfying
$$\int_{X}|f|^{2}e^{-2\varphi}dV<\infty .$$
Let $\lambda_{1}\leq\lambda_{2}\leq....\leq\lambda_{n}$ be the eigenvalues of
$\frac{i}{2\pi}\Theta_{\varphi}(L)$ and
$\widehat{\lambda}_{i}=\lambda_{i}+2\epsilon\geq\epsilon$. 
Then there exist $u$ and $v$ such that $f=\overline{\partial} u+v$ and with the following estimate
$$\int_{X}|u|^{2}e^{-2\varphi}dV+\frac{1}{2p\epsilon}\int_{X}|v|^{2}e^{-2\varphi
} dV
\leq\int_{X}\frac{1}{\widehat{\lambda}_{1}+\widehat{\lambda}_{2}+\cdots+\widehat
{\lambda}_{p}}|f|^{2}e^{-2\varphi}dV .$$
\end{prop}

\begin{proof}

Let $\omega_{1}$ be a complete Kähler metric on $X\setminus Z$ and $\omega_{\delta}=\omega+\delta\omega_{1}$
for some $\delta> 0$. 
We now do the standard $L^{2}$ estimate on $(X\setminus Z, \omega_{\delta}, L, \varphi)$.

If $s$ is a $L$-valued $(n,p)$ form in $C_{c}^{\infty}(X\setminus Z)$, then the Bochner inequality implies that:
$$\| \overline{\partial} s\|_{\delta} ^{2}+\| \overline{\partial}^{*} s\|_{\delta} ^{2}
\geq \int_{X\setminus Z} (\widehat{\lambda}_{1}+\widehat{\lambda}_{2}+\cdots+\widehat{\lambda}_{p}-2p\epsilon) |s|^{2}e^{-2\varphi}\omega_{\delta}^{n}\leqno (1)$$
where $\| s\|_{\delta} ^{2}=\int_{X} |s|^{2}e^{-2\varphi} \omega_{\delta}^{n}$.
Notice that there is an abuse of notion here. 
We calculate the norm $|u|^{2}$ by the metric (or the volum form) in the equations. 
For example, if the volum form is $\omega_{\delta}^{n}$, 
then we calculate the norm of $u$ by the metrics $\omega_{\delta}$ and $h_{0}$. 

Since $f$ is a $(n,p)$-form, the condition $\int_{X}|f|^{2}e^{-2\varphi}\omega^{n}<+\infty$ implies that 
$$f\in L^{2}(X\setminus Z, L, \varphi, \omega_{\delta})
\qquad\text{for}\hspace{5 pt} \delta> 0.$$
We write every form $s$ in the domain of the $L^{2}$ extension of $\overline{\partial}^{*}$ as $s=s_{1}+s_{2}$ with
$$s_{1}\in \Ker \overline{\partial} 
\qquad\text{and}\qquad 
s_{2} \in (\Ker \overline{\partial})^{\bot}\subset \Ker \overline{\partial}^{*} .$$
Since $f\in \Ker \overline{\partial}$, by $(1)$ we obtain
$$|\langle f, s \rangle|_{\varphi,\delta} ^{2}=|\langle f, s_{1} \rangle|_{\varphi,\delta} ^{2}$$
$$\leq \int_{X\setminus Z}\frac{1}{\widehat{\lambda}_{1}+\widehat{\lambda}_{2}+\cdots+\widehat{\lambda}_{p}}|f|^{2}e^{-2\varphi}dV_{\delta}\int_{X\setminus Z}(\widehat{\lambda}_{1}+\widehat{\lambda}_{2}+\cdots+\widehat{\lambda}_{p})|s_{1}|^{2}e^{-2\varphi}dV_{\delta}$$
$$\leq \int_{X\setminus Z}\frac{1}{\widehat{\lambda}_{1}+\widehat{\lambda}_{2}+\cdots+\widehat{\lambda}_{p}}|f|^{2}e^{-2\varphi}dV_{\delta} 
(\| \overline{\partial}^{*} s_{1}\|_{\delta} ^{2}+2p\epsilon\| \overline{\partial} s_{1}\|_{\delta} ^{2})$$
$$\leq \int_{X\setminus Z}\frac{1}{\widehat{\lambda}_{1}+\widehat{\lambda}_{2}+\cdots+\widehat{\lambda}_{p}}|f|^{2}e^{-2\varphi}dV_{\delta} 
(\| \overline{\partial}^{*} s\|_{\delta} ^{2}+2p\epsilon\| \overline{\partial} s\|_{\delta} ^{2}).$$
By the Hahn-Banach theorem, we can find $v_{\delta}, u_{\delta}$ such that
$$\langle f, s\rangle_{\delta}=\langle u_{\delta},  \overline{\partial}^{*} s\rangle_{\delta}+\langle v_{\delta}, s\rangle_{\delta}
\qquad\text{for any}\hspace{5 pt} s ,$$
and with the following estimate
$$\| u_{\delta}\|_{\delta}^{2}+\frac{1}{2p\epsilon}\| v_{\delta}\|_{\delta}^{2}\leq C\int_{X}\frac{1}{\widehat{\lambda}_{1}+\widehat{\lambda}_{2}+\cdots+\widehat{\lambda}_{p}}|f|^{2}e^{-2\varphi}\omega_{\delta}^{n}.$$
Thereofre
$$ f=\overline{\partial} u_{\delta}+v_{\delta}. \leqno (2)$$
Since the norm $\|\cdot\|_{\delta}$ for $(n,p)$-forms is increasing when $\delta\rightarrow 0$, 
we find limits
$$u=\lim_{\delta\rightarrow 0} u_{\delta} \qquad\text{and}\qquad
 v=\lim_{\delta\rightarrow 0} v_{\delta} \leqno (3)$$
satisfying
$$\| u\|_{\delta}^{2}+\frac{1}{2p\epsilon}\| v\|_{\delta}^{2}
\leq C\int_{X}\frac{1}{\widehat{\lambda}_{1}+\widehat{\lambda}_{2}+\cdots+\widehat{\lambda}_{p}}|f|^{2}e^{-2\varphi}\omega_{\delta}^{n} \leqno (4)$$
$$\leq C\int_{X}\frac{1}{\widehat{\lambda}_{1}+\widehat{\lambda}_{2}+\cdots+\widehat{\lambda}_{p}}|f|^{2}e^{-2\varphi}\omega^{n}$$
for any $\delta> 0$.
Formulas $(2)$ and $(3)$ imply that
$f=\overline{\partial} u+v$.
Let $\delta\rightarrow 0$ in $(4)$, we obtain the estimate in the proposition.

\end{proof}

\end{document}